\documentclass[12pt]{iopart}

\expandafter\let\csname equation*\endcsname\relax

\expandafter\let\csname endequation*\endcsname\relax
\usepackage{amsmath,amsfonts}
\usepackage{algorithmic}
\usepackage{etoolbox}
\usepackage{appendix}
\usepackage{algorithm}
\usepackage{array}
\usepackage[caption=false,font=normalsize,labelfont=sf,textfont=sf]{subfig}
\usepackage{textcomp}
\usepackage{stfloats}
\usepackage{url}
\usepackage{verbatim}
\usepackage{graphicx}
\usepackage{cite}
\usepackage{amsthm}
\usepackage{multirow}
\usepackage{amsopn}
\usepackage{lipsum}
\usepackage{cuted}
\usepackage{autobreak}
\usepackage{cleveref}
\usepackage{lipsum}
\AtBeginEnvironment{appendices}{\cref{section}{appendix}}
\newtheorem{theorem}{Theorem}[section]

\newtheorem{proposition}[theorem]{Proposition}

\begin{document}


\title{Convex Augmentation for Total Variation Based Phase Retrieval}


\author{Jianwei Niu$^1$, Hok Shing Wong$^1$, and Tieyong Zeng$^1$}

\address{$^1$ Department of Mathematics, the Chinese University of Hong Kong\\
}
\eads{\mailto{jwniu@math.cuhk.edu.hk}, \mailto{hswong@math.cuhk.edu.hk} and \mailto{zeng@math.cuhk.edu.hk;}}
\vspace{10pt}
\begin{abstract}
Phase retrieval is an important problem with significant physical and industrial applications. In this paper, we consider the case where the magnitude of the measurement of an underlying signal is corrupted by Gaussian noise. We introduce a convex augmentation approach for phase retrieval based on total variation regularization. In contrast to popular convex relaxation models like PhaseLift, our model can be efficiently solved by a modified semi-proximal alternating direction method of multipliers (sPADMM). The modified sPADMM is more general and flexible than the standard one, and its convergence is also established in this paper. Extensive numerical experiments are conducted to showcase the effectiveness of the proposed method. 
\end{abstract}

\vspace{2pc}

\section{Introduction}
\label{sec:intro}
It is often the case that only the magnitude of the Fourier transform of an underlying signal can be measured in real-world measurement systems. Since the phase information is missing, the procedure of reconstructing an underlying signal from its Fourier transform magnitude is usually referred to as phase retrieval (PR). Phase retrieval emerges in vast physical and industrial applications, such as optics \cite{shechtman2015phase,walther1963question}, astronomical imaging \cite{fienup1987phase,luke2002optical}, X-ray crystallography \cite{harrison1993phase,millane1990phase} and microscopy \cite{miao2008extending,shechtman2015phase}.

With the Fourier transform replaced by another arbitrarily chosen linear operator, the generalized phase retrieval problem can be formulated as \cite{kishore2015phase,waldspurger2015phase},
\begin{equation}
    \begin{aligned}
    &\text{Find } u\in\mathbb{C}^n\\
    &\text{subject to }|Au|=b,
    \end{aligned}
\label{pr-problem}
\end{equation}
where $A:\mathbb{C}^n\rightarrow\mathbb{C}^m$ is a linear operator and $b\in\mathbb{R}^m_+$.

In this paper, we focus on phase retrieval problems in a two dimensional discrete setting. We denote the $(j,k)$-entry of a 2D object $U$ by $U_{j,k}$. To make things simple, a $n_1 \times n_2$ two dimensional object $U$ can be stacked into a column vector $u$ of length $n$ in a lexicographical order, where $n=n_1*n_2$. The $j^{th}$-entry of $u$ is then denoted by $u_j$. For the lexicographical ordered signal $u$, we define the corresponding discrete Fourier transform (DFT) operator $\mathcal{F}: \mathbb{C}^{n} \rightarrow \mathbb{C}^{n}$ as
 \begin{equation}
     \mathcal{F}(u)=\left(F_{n_1\times n_1} \otimes F_{n_2\times n_2}\right) u =\operatorname{vec} \left(F_{n_1\times n_1} U F_{n_2\times n_2}\right),
     \label{Fourier}
 \end{equation}
where the symbol $\otimes$ denotes Kronecker products, $F$ is the corresponding DFT matrix and $\operatorname{vec}(\cdot)$ restack a given 2D signal into a column vector. We consider the case where the linear operator $A$ in (\ref{pr-problem}) is generated by masked Fourier measurements \cite{chen2018phase} with masks $\{D_{j}\}_{j=1}^J$ based on coded diffraction (CDP) \cite{candes2015phaseCDP}.
Explicitly, we have
\begin{equation}
A u=\left[\begin{array}{c}
\mathcal{F}\left(D_{1}  u\right) \\
\mathcal{F}\left(D_{2}  u\right) \\
\vdots \\
\mathcal{F}\left(D_{J}  u\right)
\end{array}\right],
\label{CDP}
\end{equation}
where $D_j\in\mathbb{C}^{n\times n}$ are diagonal matrices. Therefore, in all the above scenarios, $A^*A$ is also a diagonal matrix. 

Phase retrieval is extremely challenging due to its non-convexity and non-uniqueness of solution \cite{hofstetter1964construction}. In fact, it is pointed out in \cite{oppenheim1981importance} that phase information is usually more important than the magnitude information for recovering the signal from its Fourier transform. There have been plenty of works on the study of the uniqueness of a solution. There are some trivial ambiguities (non-uniqueness) such as global phase shift, conjugate inverse, and spatial shift \cite{shechtman2015phase}. For real signals of size $N$, unique recovery result with $2N-1$ random measurements is presented in \cite{balan2006signal}. As for complex signals, \cite{conca2015algebraic,fickus2014phase} extend the result by requiring $4N-4$ measurements. Unique recovery is also studied on minimum-phase signals \cite{huang2016phase} and sparse signals with non-periodic support \cite{jaganathan2013sparse}. Wong \textit{et al.} \cite{wong2020one} focused on binary signals and described a new type of ambiguities. Furthermore, Cheng \textit{et al.} \cite{cheng2021stable} considered the stability of PR problem in infinite-dimensional spaces. 

In the case of 2D images, Hayes \cite{hayes1982reconstruction} showed that $4n$ measurements are required for exact recovery of real-valued images. Cand\`{e}s \textit{et al.} \cite{candes2015phasemc} showed exact recovery result from $3n$ Fourier measurements, with the linear operator $A$ taking the form
\begin{equation}
A u=\left[\begin{array}{c}
\mathcal{F} u \\
\mathcal{F}\left(u+\mathcal{D}^{s_{1}, s_{2}} u\right) \\
\mathcal{F}\left(u-\mathbf{i} \mathcal{D}^{s_{1}, s_{2}} u\right)
\end{array}\right],
    \label{holo pattern operator}
\end{equation}
where

\begin{align}
&\left(\mathcal{D}^{s_{1}, s_{2}} u\right)_{j+(k-1)n_1}=\exp \left(\frac{2 \pi \mathbf{i} s_{1} (j-1)}{n_{1}}+\frac{2 \pi \mathbf{i} s_{2} (k-1)}{n_{2}}\right) u_{j+(k-1)n_1},
\end{align}
with $1\leq j\leq n_1$, $1\leq k\leq n_2$ and integers $s_{1}, s_{2}$ coprime to $n_{1}, n_{2},$ respectively. However, $3n$ measurements are not enough for stable exact recovery empirically and in fact $7n$ measurements are suggested. Chang \textit{et al.} \cite{chang2016phase} improved the result and proved that $3n$ measurements are sufficient both theoretically and practically when $s_1,s_2$ are both equal to $1/2$. 

The algorithms for solving the phase retrieval problem can be classified into three main categories. The first category is the 'Greedy methods', which are based on alternating projection. The pioneering work error reduction (ER) by Gerchberg and Saxton \cite{gerchberg1972practical} imposes time domain and magnitude constraints iteratively by pairs of projections. Fineup proposed variants of the ER method \cite{fienup1978reconstruction,fienup1982phase}, among which the hybrid input-output (HIO) \cite{fienup1982phase} is widely used due to its efficiency. Some other projection-based methods such as hybrid projection-reflection methods \cite{bauschke2002phase,bauschke2003hybrid}, iterated difference map \cite{elser2003phase} and relaxed averaged alternation reflection \cite{luke2004relaxed} are proposed. Moreover, saddle-point optimization is introduced in \cite{marchesini2007phase} to solve the phase retrieval problem. \cite{wen2012alternating} also proposed a similar method. However, these methods lack convergence analysis due to the projections to non-convex sets.

Another type of algorithms for random measurements based on gradient descent methods has become popular recently. Cand\`{e}s \textit{et al.} proposed the Wirtinger flow (WF) method \cite{candes2015phase}, which is a non-convex method comprising of a spectral initialization step and a gradient descent step. Unlike alternating projection methods, the convergence of WF can be guaranteed. The Wirtinger flow approach is further improved by the work truncated Wirtinger flow (TWF) \cite{chen2017solving}. Incremental methods are also proposed to solve (\ref{pr-problem}), for example the Kaczmarz methods \cite{li2015phase,wei2015solving}. Incremental versions of TWF is introduced in \cite{zhang2016provable}. To reduce complexity, reshaped Wirtinger flow (RWF) and its incremental version IRWF with a lower-order loss function are presented in \cite{zhang2016reshaped}. 

The final category is the convex relaxation method. Since phase retrieval constraints are quadratic, semidefinite programming (SDP) techniques can be applied to solve such problems \cite{goemans1995improved,so2007approximating}. In such approaches, the variable is often 'lifted' into a higher-dimensional space, and the phase retrieval problem is converted to a tractable SDP. For example, PhaseLift \cite{candes2013phaselift} reformulates the phase retrieval problem as a rank-1 minimization, which is subsequently relaxed to a convex nuclear norm minimization problem. Waldspurger \textit{et al.} proposed another convex method PhaseCut \cite{waldspurger2015phase}, which separates the phase and magnitude through complex semidefinite programming. However, due to a large number of variables, the convex relaxation methods are computationally demanding and impractical.

In this paper, we consider phase retrieval problems with the magnitude of measurements corrupted by Gaussian noise. That is, the measurement $g$ satisfies 
\begin{equation}
    g = |Au|+\xi,
    \label{noisy model}
\end{equation}
where $\xi$ is an additive Gaussian white noise. Since non-convex models can get stuck at a local minimum and the convergence result can usually be shown on subsequences only, we look for a convex augmentation model in this paper. Furthermore, total variation (TV) regularization \cite{cai2012image,chambolle2004algorithm,rudin1992nonlinear} has been widely used in different tasks of image processing. It is also shown in \cite{lustig2007sparse,lustig2008compressed} that TV regularization efficiently recovers signals from incomplete information. On the other hand, TV regularization can also be used to deal with phase retrieval problems with noisy measurements. For example, TV regularization was introduced to variational models to suppress noise in \cite{chang2018variational,zhang2005total}. Chang \textit{et al.} \cite{chang2018total} also proposed a TV-based model for phase retrieval with Poisson noise in measurements. Motivated by this, we also consider a TV regularized model in this paper.

Our contribution is twofold. First, by approximating the data term with a convex augmentation and restricting the domain, we proposed a convex augmentation model to deal with magnitude measurements with additive Gaussian noise. Based on this convex augmentation technique, it is possible to further improve the phase retrieval model by replacing the TV regularizer with other regularizers. Second, based on the alternating direction method of multipliers (ADMM) \cite{boyd2011distributed,chan2013constrained,glowinski1989augmented,wu2010augmented} which is often used to solve TV-regularized minimization, we solve the proposed model by a modified semi-proximal ADMM (sPADMM) \cite{fazel2013hankel,li2015convergent}. Due to the multiple linear constraints in the proposed model, the standard sPADMM does not match exactly with the proposed model. Therefore, we modify the augmented Lagrangian and design a more flexible sPADMM that suits well with the model. The existence of the solution and convergence result of the modified sPADMM will also be established. The modified sPADMM presented here can also be useful for a wide range of problems with multiple linear constraints. Unlike convex models like PhaseLift, the proposed model can be efficiently solved by the modified sPADMM. Extensive numerical results also demonstrate the outstanding performance of the proposed method.

The paper is organized as follows. We first recall some notations in \Cref{sec:notations}. In \Cref{sec:Proposed model}, a convex augmentation total variation-based model for phase retrieval problems with Gaussian noise is introduced. The convexity of the proposed method and the existence of the solution are presented. The sPADMM-based algorithm for solving the proposed model is introduced in \Cref{sec:alg}. The convergence of the algorithm will also be presented. Extensive numerical experiments are conducted in \Cref{sec:experiments} to demonstrate the proposed method's effectiveness. Conclusions and future works are in \Cref{sec:conclusions}.

\section{Notations}
\label{sec:notations}

For the $K$ dimensional complex Hilbert space $\mathbb{C}^K$ and $\mu, \nu \in \mathbb{C}^K$, the inner product is defined by $\langle \mu, \nu\rangle= \sum_{j=1}^{K} \mu_{j} \bar{\nu}_{j}$, where  $\bar{(\cdot)}$ means the conjugate of a complex number, vector or matrix. For matrices, we also use ${(\cdot)}^T$ and ${(\cdot)}^{*}$ to denote their transpose and Hermitian transpose, respectively. Given any Hermitian positive definite matrix $S \in \mathbb{C}^{K\times K}$, we denote $\langle \mu,S\nu\rangle$ by $\langle \mu, \nu\rangle_{S}$. In addition, we denote $\|\cdot\|=\sqrt{\langle \cdot , \cdot \rangle}$ as the $l^{2}$-norm and $\|\cdot\|_S=\sqrt{\langle \cdot , \cdot \rangle_S}$ as the induced $S$ norm. The $l^1$-norm w.r.t. $\mu$ is defined by $\|\mu\|_1=\sum_{j=1}^{K} |\mu_{j}|$.  For the sake of brevity, we may use $\operatorname{Re}(\cdot)$ or $(\cdot)_{re} $ to represent the real part of a number, vector or matrix, and $\operatorname{Im}(\cdot)$ or $(\cdot)_{im} $ to represent the imaginary part of a number, vector or matrix.

Suppose $u$ is the lexicographically restacked column vector representation of an image $U$. Following the notations in \Cref{sec:intro}, we define the difference operators $\nabla_x$ and $\nabla_y$ for $U$ as
\begin{equation*}
\begin{array}{ll}
(\nabla_x U)_{j,k} & =\left\{\begin{array}{ll}
U_{j+1, k}-U_{j, k} & \text { if } 1 \leq j\leq n_1-1, \\
0 & \text { if } j=n_1,
\end{array}\right. \\
(\nabla_y U)_{j,k} & =\left\{\begin{array}{ll}
U_{j, k+1}-U_{j, k} & \text { if } 1 \leq k\leq n_2-1, \\
0 & \text { if } k=n_2.
\end{array}\right.
\end{array}
\end{equation*} 

 We denote the discrete gradient of $u$ as
\begin{equation*}
    \nabla u = \left[\begin{array}{l}
\nabla_{x} u \\
\nabla_{y} u
\end{array}\right],
\end{equation*}
where $\nabla_x u=\operatorname{vec}(\nabla_x U)$ and $\nabla_y u=\operatorname{vec}(\nabla_y U)$. We further define the anisotropic TV operator \cite{shi2013efficient} as 
\begin{equation}
   \operatorname{TV}(u)=\|\nabla u\|_{1}=\sum_{j=1}^{n} \left(\left|(\nabla_x u)_j\right|+\left|(\nabla_y u)_j\right|\right).
\end{equation}

Throughout this paper, we will no longer distinguish linear operators and their corresponding matrix representations. Moreover, the divergence operator and the Laplacian are denoted by $\operatorname{div}=-\nabla^*$ and $\Delta=\operatorname{div} \cdot \nabla=-\nabla^*\nabla$, respectively. 
We define the index set $\Omega$ of masked Fourier measurements to be 
$$
\Omega=\{ 1,2, \ldots, Jn  \}, 
$$
where $J$ denotes the number of masks, and we let $\Gamma\subseteq\Omega$ be the under-sampling set.

\section{Proposed model}                                                   
\label{sec:Proposed model}
A simple total variation based model for phase retrieval can be given by
\begin{equation}
\min_{u \in C^{n}} \quad E_{0}(u)=\lambda \operatorname{TV}(u)+\sum_{j \in \Gamma} \left(|A u|_j-g_j \right) ^{2},
\label{TV model}
\end{equation}
where $g$ is the measured data and $\lambda>0$ is a fixed parameter.

To further improve the above model, we first perform a warm-up step. An initial point $\hat{u}$ is computed by running a few iterations of a simple yet efficient PR algorithm like error reduction (ER). We denote
\begin{equation}
    \hat{u}=\operatorname{Init}(A,g)
    \label{initial}
\end{equation}
to be the initial point. We also make $\hat{z}=A\hat{u}$ for future use.

By taking account of the initial point, we consider the following model
\begin{align}
\min_{u \in C^{n}} \quad E_{1}(u)&=\lambda \operatorname{TV}(u)+\sum_{j \in \Gamma} \left(|A u|_j-g_j \right) ^{2}\nonumber +\eta\sum_{j\in\Omega}|(Au)_j-(A\hat{u})_j|^2,
\label{TV-ini model}
\end{align}
where $\eta>0$ is another parameter.

However, the above model is not convex. Since non-convex models can often be stuck at a local minimum, we consider convex augmentation of the above model in what follows.

Since the TV term is convex, we focus on \begin{equation}
    \widetilde{G}_1(z)=\sum_{j\in \Gamma}(g_j-|z_j|)^2+\eta\sum_{j\in\Omega}|z_j-\hat{z}_j|^2,
\end{equation}
where $z=Au$.  

To improve the smoothness and convexity of $\widetilde{G}_1$, we introduce the following approximation of $\widetilde{G}_1$
\begin{equation}
    \widetilde{G}_2(z) = \sum_{j\in \Gamma}(g_j-\sqrt{|z_j|^2+\delta_j})^2+\eta\sum_{j\in\Omega}|z_j-\hat{z}_j|^2,
\end{equation}
where $\delta_j>0$ is a small parameter.

Unfortunately, the above function is still not convex everywhere. In order to get a convex augmentation model, we need to locate the region where $\widetilde{G}_2$ is convex. 

Since $\widetilde{G}_2$ is separable with respect to each component of $z$, we then focus on the function

\begin{equation}
G_j(z_j)=\begin{cases}
\begin{aligned}& \left(g_j-\sqrt{|z_j|^{2}+\delta_j}\right)^{2}  +\eta \left|z_j-\hat{z}_j\right|^{2}\end{aligned}, & \text { if } j \in \Gamma \\
\eta \left|z_j-\hat{z}_j\right|^{2}, & \text { if } j \in \Omega \setminus \Gamma
\end{cases}.
\label{fn:Gj}
\end{equation}

If $j\in\Omega\setminus \Gamma$, then $G_j$ is clearly convex. By writing $z=z_{re}+z_{im}\mathbf{i}$ and $\hat{z}=\hat{z}_{re}+\hat{z}_{im}\mathbf{i}$ and dropping the index for the moment, we consider the function
\begin{align}
    G(z_{re},z_{im})& = (g-\sqrt{z_{re}^2+z_{im}^2+\delta})^2\nonumber +\eta(z_{re}-\hat{z}_{re})^2+\eta(z_{im}-\hat{z}_{im})^2.
\end{align}

The first order partial derivatives of $G(z_{re},z_{im})$ are
\begin{equation}
\begin{aligned}
\frac{\partial G}{\partial z_{re}}&=\frac{-2 g z_{re}}{\sqrt{z_{re}^{2}+z_{im}^{2}+\delta}}+\left(2+2 \eta \right) z_{re}-2 \eta \hat{z}_{re}, \\
\frac{\partial G}{\partial z_{im}}&=\frac{-2 g z_{im}}{\sqrt{z_{re}^{2}+z_{im}^{2}+\delta}}+\left(2+2 \eta \right) z_{im}-2 \eta \hat{z}_{im}. 
\end{aligned}
\end{equation}

The corresponding Hessian matrix is
\begin{equation}
\begin{aligned}
\nabla^2 G &=\left[\begin{array}{cc}
\displaystyle\frac{\partial^{2} G}{\partial z_{re}^{2}} & \displaystyle\frac{\partial^{2} G}{\partial z_{re} \partial z_{im}}  \\
\displaystyle\frac{\partial^{2} G}{\partial z_{im} \partial z_{re}} & \displaystyle\frac{\partial^{2} G}{\partial z_{im}^{2}} 
\end{array}\right]\\
&=\left[\begin{array}{cc}
\displaystyle\frac{-2g\left(z_{im}^{2}+\delta\right)}{\left(r+\delta\right)^{\frac{3}{2}}} & \displaystyle\frac{2 g z_{re} z_{im}} {\left(r+\delta\right)^{\frac{3}{2}}} \\
\displaystyle\frac{ 2 g z_{re} z_{im}}{\left(r+\delta\right)^{\frac{3}{2}}} & \displaystyle\frac{-2 g\left(z_{re}^{2}+\delta\right)}{\left(r+\delta\right)^{\frac{3}{2}}}
\end{array}\right]+(2+2 \eta)\mathbf{I},
\end{aligned}  
\label{Hessian}
\end{equation}
where $r=z_{re}^2+z_{im}^2$ and $\mathbf{I}$ denotes the identity matrix.

\begin{proposition}
For any $\delta>0$, the Hessian matrix $\nabla^2 G$ is positive definite when 
\begin{equation}
(z_{re},z_{im}) \in \begin{cases}
\left\{(z_{re},z_{im})\ | \ z_{re}^2+z_{im}^2> \frac{4}{3}\frac{g^2}{(1+\eta)^2} \right\}, &g>0\\
\mathbb{R}^2, &g \leq 0
\end{cases}.
\label{zregion}
\end{equation}
\label{prop:p1}
\end{proposition}
\begin{proof}
Since ${\left(z_{re}^{2}+z_{im}^{2}+\delta\right)^{\frac{3}{2}}}>0 $, we  consider a new matrix 

\begin{equation}
M= \left(r+\delta\right)^{\frac{3}{2}} \cdot \nabla^{2}G.
\end{equation}

We now compute the trace and determinant of $M$:
\begin{equation}
\operatorname{trace}(M)=(4+4 \eta)\left(r+\delta \right)^{\frac{3}{2}}-2g \left(r+2 \delta\right)
\label{trace}
\end{equation}
and
\begin{align}
\operatorname{det}(M)=&4(1+\eta)(r+\delta)^{\frac{3}{2}}\left[(1+\eta)(r+\delta)^{\frac{3}{2}}-g(r+2 \delta)\right]\nonumber+4\delta r g^{2}+4\delta^{2} g^{2}
\label{det}.
\end{align}
Since $M$ is a $2\times 2$ real symmetric matrix, $M$ is positive definite if and only if both $\operatorname{trace}(M)$ and $\operatorname{det}(M)$ are positive.

For $g \leq 0$, $\operatorname{trace}(M)$ and  $\operatorname{det}(M)$ are directly positive from which we can derive the result when $g\leq 0$ in (\ref{zregion}). 

For $g>0$, we define a function $h$ with respect to $\delta$ as follows,
\begin{equation}
h(\delta)=(1+\eta)^2(r+\delta)^{3}-g^2(r+2 \delta)^2 ,\quad \delta > 0.
\end{equation}
 Obviously, $h(\delta)>0$ can guarantee  $\operatorname{trace}(M)$ and  $\operatorname{det}(M)$  positive. 
Take the first and second derivative of $h(\delta)$, and we can derive
\begin{equation}
\begin{aligned}
h'(\delta)& =3(1+\eta)^{2}(r+\delta)^{2}-4 g^{2}(r+2 \delta) 
\end{aligned}
\end{equation}
and
\begin{equation}
\begin{aligned}
h''(\delta)& =6(1+\eta)^{2}(r+\delta)-8g^{2}.
\end{aligned}
\end{equation}

If $h''(0)>0$ and $h'(0)>0$, $h(\delta)$ is increasing for $\delta >0$. Therefore, a sufficient condition making $h(\delta)$ positive can be given by
\begin{equation}
 h''(0)>0, \  h'(0)>0 \ \text{and} \ h(0)>0.
\end{equation}
The corresponding $r$ will be chosen as 
\begin{equation}
\ r> \frac{4}{3}\frac{g^2}{(1+\eta)^2}.
\end{equation} 
\end{proof}

\begin{figure}[h]
\centering
\includegraphics[width=\linewidth]{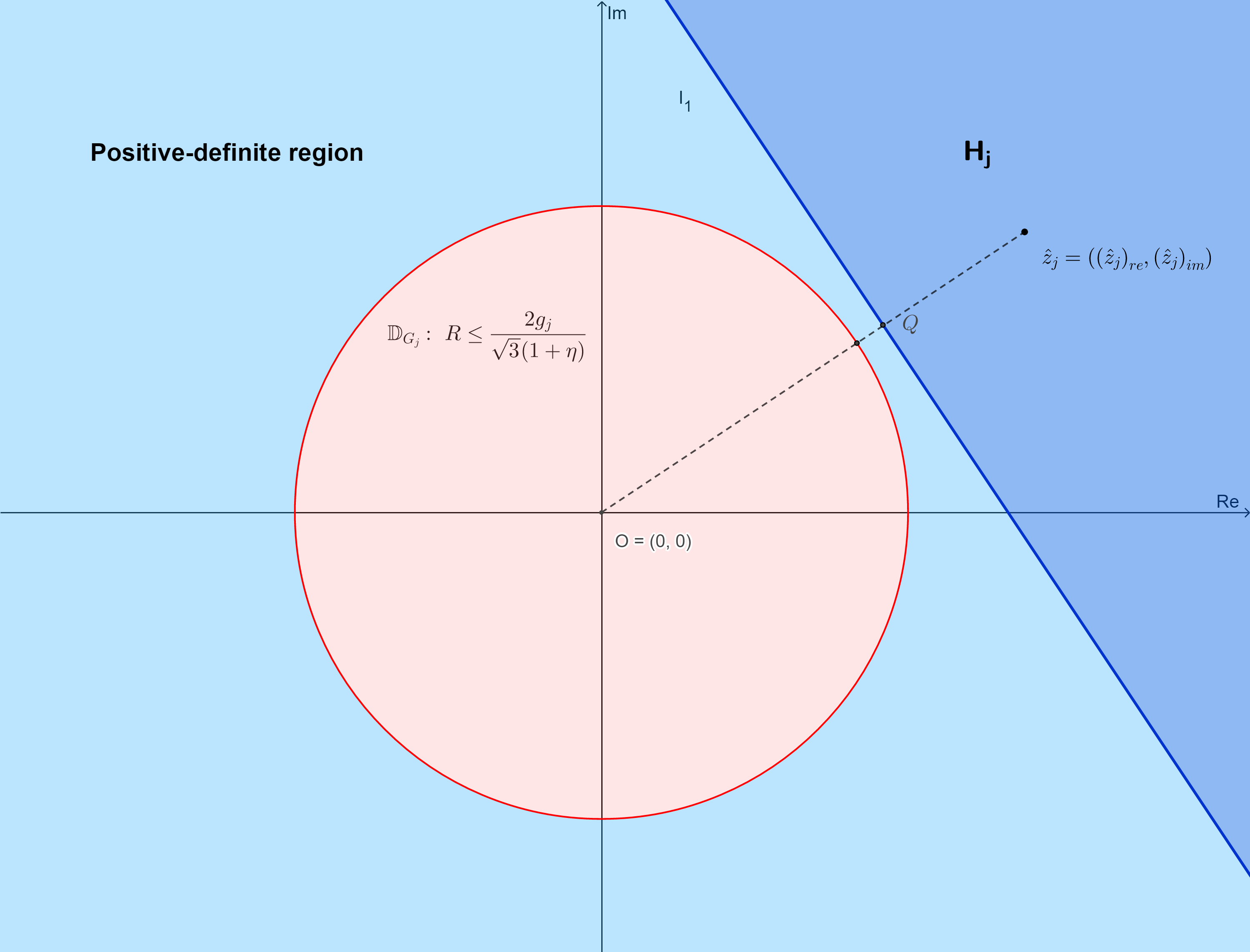} 
\caption{The region where $\nabla^2 G_j$ is positive definite. When $g_j>0$, the positive definite region is not convex.} 
\label{trustregion} 
\end{figure}

\subsection{Closed convex domain}
\Cref{prop:p1} provides a region where $G_j$ is convex. However, this region is not convex. As a final step to reach a convex augmentation model, we need to restrict the domain of $G_j$ further. Let $\mathbb{D}_{G_j}$ denote the complement of the convex region in \Cref{prop:p1}, which is a closed disk. We assume each component of the initial point $\hat{z}$ is in the restricted domain. With this assumption, we can consider the halfspace $H_j$ defined by the tangent plane at the projection of $\hat{z}_j$ to the disk $\mathbb{D}_{G_j}$. In practice, for numerical stability consideration, we slightly shift the tangent plane towards $\hat{z}_j$, and $H_j$ is then defined by this hyperplane. The positive definite region and $H_j$ are illustrated in \Cref{trustregion}. 

Note that there may exists $j$ such that $\hat{z}_j$ belongs to $\mathbb{D}_{G_j}$, which is 
\begin{equation}
   (\hat{z}_j)_{re}^2+(\hat{z}_j)_{im}^2 \leq \frac{4}{3}\frac{g_j^{2}}{(1+\eta)^{2}}. 
\label{disk_con}
\end{equation}
However, with a moderate $\eta$, say $\eta=1$, (\ref{disk_con}) means
\begin{equation}
   (\hat{z}_j)_{re}^2+(\hat{z}_j)_{im}^2 \leq \frac{g_j^{2}}{3}. 
\end{equation}
Since $\hat{z}$ is assumed to be an initial approximated solution, it is reasonable to assume $|\hat{z}|\approx g$. In this case, we can modify $\hat{z}_j$ such that it belongs to the restricted convex region. 

\subsection{Relaxed convex augmentation model}
Recall the convex domain of each $G_j$ is denoted by $H_j$. By restricting the domain, we finally get the following convex augmentation model
\begin{equation}
    \begin{aligned}
        \min_{u \in C^{n}}  E(u)&=\lambda \operatorname{TV}(u)+\sum_{j \in \Gamma} \left(g_j-\sqrt{|(Au)_j|^2+\delta_j} \right) ^{2}\\ &\ +\eta\sum_{j\in\Omega}|(Au)_j-(A\hat{u})_j|^2+\sum_{j\in \Gamma}\mathbb{I}_{H_j}((Au)_j).
    \end{aligned}
\label{convex model}
\end{equation}

In the following, we show the existence of solutions of (\ref{convex model}) under mild assumptions.
\begin{proposition}
Assume that there is a positive number $\beta$ such that 
\begin{equation}
    \beta\|u\|\leq\|Au\|_{2,\Gamma},
\label{coer-assump}
\end{equation}
where $\|z\|_{2,\Gamma}=\sqrt{\sum_{j\in \Gamma}|z_j|^2}$. Then (\ref{convex model}) admits a solution $u^*$.
\end{proposition}
\begin{proof}
Note that $E(u)$ is convex and lower semicontinuous. By (\ref{coer-assump}), $\|Au\|_{2,\Gamma}$ goes to infinity as $\|u\|$ goes to infinity. Hence $E$ is also coercive. By standard result in convex analysis \cite{badiale2010semilinear}, (\ref{convex model}) admits a solution $u^*$.
\end{proof}

\section{Algorithm implementation}
\label{sec:alg}
In what follows, we introduce a semi-proximal alternating direction method of multipliers (sPADMM) to solve (\ref{convex model}). We first rewrite the convex augmentation model (\ref{convex model}) by
\begin{equation}
\begin{array}{l}
 \begin{aligned}
\min _{u ,z,p} \quad E_{\delta}(u, z, p) &=\lambda\|p\|_{1}+\sum_{j \in \Gamma}\left(g_j-\sqrt{|z_j|^{2}+\delta_j}\right)^{2} \\
&\ + \eta \sum_{j \in \Omega}\left|z_j-\hat{z}_j\right|^{2}+\sum_{j\in \Gamma}\mathbb{I}_{H_j}(z_j)
\end{aligned}
\\
\text { such that } \quad z=A u, \quad p=\nabla u,
\end{array}
\label{model:relaxed discrete model}
\end{equation}
where $\mathbb{I}_{H_j}(\cdot)$ is the indicator function of the convex domain $H_j$ of $G_j$ which is defined in (\ref{fn:Gj}). According to previous analyses, the objective function $E_{\delta}(u, z, p)$ is proper, closed and convex.

Since the reformulated convex augmentation model (\ref{model:relaxed discrete model}) contains two linear constraints while the original sPADMM method only contains one, we slightly modify the sPADMM algorithm to fit with our proposed model. The augmented Lagrangian of (\ref{model:relaxed discrete model}) is defined by
\begin{equation}
\begin{aligned}
&L_{\alpha, \gamma}(u, z, p, w, q)\\ &\ =\lambda\|p\|_{1}+\sum_{j \in \Gamma}\left(g_j-\sqrt{|z_j|^{2}+\delta_j}\right)^{2} +\eta \sum_{j \in \Omega}\left|z_j-\hat{z}_j\right|^{2}  \\
&\ + \sum_{j \in \Gamma} \mathbb{I}_{H_j} \left( z_j\right) +\operatorname{Re}(\langle z-A u, w\rangle)+\operatorname{Re}(\langle p-\nabla u, q\rangle) \\
&\ +\frac{\alpha}{2}\|z-A u\|^{2}+\frac{\gamma}{2}\|p-\nabla u\|^{2}.
\end{aligned}
\label{Lmain-equation}
\end{equation}
The corresponding algorithm for solving (\ref{model:relaxed discrete model}) is then shown in \Cref{alg:MpADMMcam}.  

\begin{algorithm}[h]
\caption{Modified semi-proximal ADMM for convex augmentation model (\ref{convex model})}
\label{alg:MpADMMcam}
{\bf Step 0}. Input $ \left(u^{0}, z^{0}, p^{0},w^{0},q^{0}\right) \in \mathbb{C}^n \times \mathbb{C}^{Jn} \times \mathbb{C}^{2n} \times \mathbb{C}^{Jn} \times  \mathbb{C}^{2n} .$ 

{\bf Step 1}. Set 
\begin{equation}
\left\{\begin{aligned}
u^{k+1} &= \underset{u \in \mathbb{C}^n}{\arg \min } \ \operatorname{Re}(\langle z^k-A u, w^k\rangle) 
+\operatorname{Re}(\langle p^k-\nabla u, q^k\rangle)\\
&\ +\frac{\alpha}{2}\|z^k-A u\|^{2}+\frac{\gamma}{2}\|p^k-\nabla u\|^{2}+\left\|u-u^{k}\right\|_{S_1}^{2}, \\
z^{k+1}&=\underset{z \in \mathbb{C}^{Jn}}{\arg \min } \ \sum_{j \in \Gamma}\left(g_j-\sqrt{|z_j|^{2}+\delta}\right)^{2}+\sum_{j \in \Gamma} \mathbb{I}_{\Pi(\hat{z}_j)} \left( z_j\right)\\& \ +\left\|z-z^{k}\right\|_{S_2}^{2}+\eta \sum_{j \in \Omega}\left|z_j-\hat{z}_j\right|^{2}+\frac{\alpha}{2}\|z-Au^{k+1}\|^{2}\\& \ +\operatorname{Re}(\langle z-Au^{k+1}, w^k\rangle),\\
p^{k+1}&=\underset{p \in \mathbb{C}^{2n}}{\arg \min } \ \lambda\|p\|_{1}
+\operatorname{Re}(\langle p-\nabla u^{k+1}, q^k\rangle)\\& \ +\frac{\gamma}{2}\|p-\nabla u^{k+1}\|^{2}+\left\| p-p^k \right\|_{S_3}^2,\\
w^{k+1}&=w^k-\tau\alpha\left(z^{k+1}-Au^{k+1}\right),\\
q^{k+1}&=q^k-\tau\gamma\left(p^{k+1}-\nabla u^{k+1} \right),
\end{aligned}\right.
\label{alg:subproblem}
\end{equation}
where $\alpha>0$, $\gamma>0$ are the penalty parameters in (\ref{Lmain-equation}) , $\tau \in(0,(1+\sqrt{5}) / 2)$ is the step length, and $S_1$, $S_2$ and $S_3$ are self-adjoint positive semidefinite, not necessarily positive definite, operators on $\mathbb{C}^{n}$, $\mathbb{C}^{Jn}$ and $\mathbb{C}^{2n}$  respectively.

{\bf Step 2}. If a termination criterion is not met, go to Step 1.
\label{algorithm}
\end{algorithm}

Next, we demonstrate how to solve each of the three subproblems in (\ref{alg:subproblem}) respectively.
\subsection{u-subproblem}
With a proximal term in $u$ added to the augmented Lagrangian (\ref{Lmain-equation}), the $u$ subproblem can be written as
\begin{equation}
\begin{aligned}
u^{k+1} &= \underset{u \in \mathbb{C}^n}{\arg \min } \ \operatorname{Re}(\langle z^k-Au, w^k\rangle) 
+\operatorname{Re}(\langle p^k-\nabla u, q^k\rangle)\\
&\ +\frac{\alpha}{2}\|z^k-Au\|^{2}+\frac{\gamma}{2}\|p^k-\nabla u\|^{2}+\left\|u-u^{k}\right\|_{S_1}^{2},
\end{aligned}
\label{u1 min}
\end{equation}
where $S_1$ is a positive semidefinite Hermitian operator acting as the
proximity term. (\ref{u1 min}) can be further simplified as the following
equivalent form
\begin{equation}
\begin{aligned}
u^{k+1} = \underset{u \in \mathbb{C}^n}{\arg \min }  &\frac{\alpha}{2}\|z^k-A u+\frac{w^k}{\alpha}\|^{2}+\frac{\gamma}{2}\|p^k
+\frac{q^k}{\gamma}-\nabla u\|^{2}+\left\|u-u^{k}\right\|_{S_1}^{2},
\end{aligned}
\label{u2 min}
\end{equation}
where $A$ and $u$ are rewritten as  $A=A_{r e}+A_{i m} \mathbf{i}$ and  $u=u_{re}+ u_{im} \mathbf{i} $ respectively. Then we can compute that $Au=A_{r e} u_{r e}-A_{i m} u_{i m}+\left(A_{r e} u_{i m}+A_{i m} u_{r e}\right)\mathbf{i}$. Taking $\widetilde{w}^{k}=z^{k}+\frac{w^{k}}{\alpha}$ and $\widetilde{p}^{k}=p^{k}+\frac{q^{k}}{\gamma}$, (\ref{u2 min}) is equivalent to 
\begin{equation}
\begin{aligned}
 u^{k+1} =\underset{u \in \mathbb{C}^n}{\arg \min }  &\frac{\alpha}{2}\| \widetilde{w}^{k}-Au\|^{2}+\frac{\gamma}{2}\|\widetilde{p}^{k}-\nabla u\|^{2}+\left\|u-u^{k}\right\|_{S_1}^{2}.
\end{aligned}
\label{u3 min}
\end{equation}
Consider real and imaginary parts separately, and we get
\begin{equation}
\begin{aligned}
&\frac{\alpha}{2}\| \widetilde{w}^{k}-Au\|^{2}+\frac{\gamma}{2}\|\widetilde{p}^{k}-\nabla u\|^{2}+\left\|u-u^{k}\right\|_{S_1}^{2}\\
& =\frac{\alpha}{2}\| A_{r e} u_{r e}-A_{i m} u_{i m}-\operatorname{Re}(\widetilde{w}^{k})\|^{2}+\frac{\gamma}{2}\|\operatorname{Im}(\widetilde{p}^{k})-\nabla u_{im}\|^{2}\\
&\ +\frac{\alpha}{2}\| A_{r e} u_{im}+A_{i m} u_{re}-\operatorname{Im}(\widetilde{w}^{k})\|^{2}+\frac{\gamma}{2}\|\operatorname{Re}(\widetilde{p}^{k})-\nabla u_{re}\|^{2}\\
&\ -2\left\langle u_{re}-\operatorname{Re}(u^k),\operatorname{Im}(S_1)(u_{im}-\operatorname{Im}(u^k))\right\rangle\\
&\ +\left\|u_{re}-\operatorname{Re}(u^{k})\right\|_{S_1}^{2}+\left\|u_{im}-\operatorname{Im}(u^{k})\right\|_{S_1}^{2}. 
\end{aligned}
\label{u4}
\end{equation}
Since (\ref{u4}) is convex in $u$, thus by taking the gradient of (\ref{u4}), we obtain that
\begin{equation}
\begin{aligned}
&\left[\begin{array}{cc}
  \mathcal{B}_{11} & -\mathcal{B}_{21}\\
\mathcal{B}_{21} & \mathcal{B}_{22}
\end{array}\right]
\cdot \left[\begin{array}{c}
u_{r e} \\
u_{i m}
\end{array}\right]\\
&=\left[\begin{array}{c}
 \alpha\left(\operatorname{Re}\left(A^{*} \widetilde{w}^{k}\right)\right)-\gamma\left(\nabla \cdot \operatorname{Re}\left(\widetilde{p}^{k}\right)\right)+2\operatorname{Re}(S_1u^k) \\
 \alpha\left(\operatorname{Im}\left(A^{*} \widetilde{w}^{k}\right)\right)-\gamma\left(\nabla \cdot \operatorname{Im}\left(\widetilde{p}^{k}\right)\right)+2\operatorname{Im}(S_1u^k)
\end{array}\right],
\end{aligned}
\label{d equation 1}
\end{equation}
where $\mathcal{B}_{11}=\mathcal{B}_{22}=\alpha\operatorname{Re}(A^* A)-\gamma \Delta+2 \operatorname{Re} (S_1)$, $\mathcal{B}_{21}=\alpha\operatorname{Im}(A^* A)+2\operatorname{Im}(S_1)$.

With some mild assumptions, we can show that the $u$-subplobrem (\ref{u1 min}) has a unique solution.
\begin{proposition}
If we take $\operatorname{Im}(S_1)=-\frac{\alpha}{2}\operatorname{Im}(A^* A)$, then the linear equations (\ref{d equation 1}) admit a unique solution.
\label{proposition:unique}
\end{proposition}
\begin{proof}
Given $\operatorname{Im}(S_1)=-\frac{\alpha}{2}\operatorname{Im}(A^* A)$, we define the coefficient matrix
\begin{equation}
\mathcal{B}:=\left[\begin{array}{cc}
  \mathcal{B}_{11} & 0\\
0 & \mathcal{B}_{22}
\end{array}\right],
\end{equation}
where $\mathcal{B}_{11},\mathcal{B}_{22}$ are defined as above.

We will show that the linear operator $\mathcal{B}$ is nonsingular. To achieve this goal, we calculate
$$
\begin{aligned}
&\left\langle\left(\begin{array}{l}
u_{re} \\
u_{im}
\end{array}\right),\mathcal{B} \left(\begin{array}{l}
u_{re} \\
u_{im}
\end{array}\right)\right\rangle\\
&=\left\langle u_{re},\alpha \operatorname{Re}\left(A^* A\right) u_{re}-\gamma \Delta u_{re}+2\operatorname{Re}(S_1)u_{re}\right\rangle\\
&\ +\left\langle u_{im}, \alpha \operatorname{Re}\left(A^* A\right) u_{im}-\gamma \Delta u_{im}+2\operatorname{Re}(S_1)u_{im}\right\rangle\\
&= \alpha\left\langle u_{re},\operatorname{Re}\left(A^* A\right) u_{re} \right\rangle + 2\left\langle u_{re}, \operatorname{Re}(S_1) u_{re} \right\rangle \\
&\ +\alpha\left\langle u_{im},\operatorname{Re}\left(A^* A\right) u_{im} \right\rangle + 2\left\langle u_{im}, \operatorname{Re}(S_1) u_{im} \right\rangle \\
& \ +\gamma\left(\left\langle-\Delta u_{re}, u_{re}\right\rangle
+\left\langle-\Delta u_{im}, u_{im}\right\rangle\right),
\end{aligned}
$$
where the Laplace operator $\Delta$ is negative definite, $\operatorname{Re}\left(A^{*} A\right)=A_{re}^{T} A_{re}+A_{im}^{T} A_{im} $  and $\operatorname{Re} (S_1)= \left(S_1+\Bar{S_1}\right)/2 $ are positive semidefinite. Then it follows that $\mathcal{B}$ is also positive definite and hence non-singular. Finally, the corresponding solution of (\ref{d equation 1}) can be written as
\begin{equation}
\begin{aligned}
u^{k+1}=&\left(\alpha \operatorname{Re}\left(A^{*} A\right)-\gamma \Delta+2 \operatorname{Re}(S_1)\right)^{-1} \cdot \left(\alpha A^{*} \widetilde{w}^{k}-\gamma\nabla \cdot \left(\widetilde{p}^{k}\right)+2S_1 u^k\right).
\end{aligned}
\end{equation}
\end{proof}

CDP measurements are considered in our experiments. For such patterns, $A^*A$ is a real diagonal matrix with different diagonal entries. In particular, $\operatorname{Im}(A^* A) = 0$. By taking a real positive semidefinite $S_1$, we can get a new linear equation
\begin{equation}
\begin{aligned}
&\mathcal{B} \left[\begin{array}{c}
u_{r e} \\
u_{i m}
\end{array}\right]
=\left[\begin{array}{c}
 \alpha\left(\operatorname{Re}\left(A^{*} \widetilde{w}^{k}\right)\right)-\gamma\left(\nabla \cdot \operatorname{Re}\left(\widetilde{p}^{k}\right)\right)+2\operatorname{Re}(S_1u^k) \\
 \alpha\left(\operatorname{Im}\left(A^{*} \widetilde{w}^{k}\right)\right)-\gamma\left(\nabla \cdot \operatorname{Im}\left(\widetilde{p}^{k}\right)\right)+2\operatorname{Im}(S_1u^k)
\end{array}\right],
\end{aligned}
\label{d equation 2}
\end{equation}
where $\mathcal{B}$ is defined in \Cref{proposition:unique}. 

Unfortunately, the FFT can not be directly applied to (\ref{d equation 2}). However, since the coefficient matrix $\mathcal{B}$ is sparse and symmetric, one can use the conjugate gradient (CG) method or biconjugate gradient (BICG) to solve (\ref{d equation 2}) efficiently. 
 
\subsection{z-subproblem}
 The $z$-subproblem can be written as
 \begin{equation}
 \begin{aligned}
 z^{k+1} =\underset{z \in \mathbb{C}^{Jn}}{\arg \min } &\sum_{j \in \Gamma}\left(g_j-\sqrt{|z_j|^{2}+\delta}\right)^{2}+\sum_{j \in \Gamma} \mathbb{I}_{H_j} \left( z_j\right)+\left\|z-z^{k}\right\|_{S_2}^{2} +\eta \sum_{j \in \Omega}\left|z_j-\hat{z}_j\right|^{2}\\
 &+\frac{\alpha}{2}\|z-Au^{k+1}\|^{2}+\operatorname{Re}(\langle z-Au^{k+1}, w^k\rangle).
 \end{aligned} 
 \label{z1}
 \end{equation}
 
 Considering the case where $S_2$ is diagonal, for convenience, we let
\begin{equation}
    \operatorname{T}_j=\frac{\alpha \left(Au^{k+1}\right)_j+2\eta \hat{z}_j-w^{k}_j+2 (S_2)_j z^{k}_j}{\alpha+2\eta+2 (S_2)_j}
\end{equation}
and
\begin{equation}
    \operatorname{coef}_j=\frac{\alpha}{2}+\eta+(S_2)_j,
\end{equation} 
where $(S_2)_j$ for $ j \in \Omega$ is the $j^{th}$ diagonal element of $S_2$.

We can reformulate (\ref{z1}) as
 \begin{equation}
 \begin{aligned}
 z^{k+1}=\underset{z \in \mathbb{C}^{Jn}}{\arg \min } &\sum_{j \in \Gamma}\left(g_j-\sqrt{|z_j|^{2}+\delta}\right)^{2}+\sum_{j \in \Gamma} \mathbb{I}_{H_j} \left( z_j\right)  + \sum_{j\in \Omega }\operatorname{coef}_j \cdot \left|z_j-T_j\right|^{2}.
  \end{aligned}
 \label{z2}
 \end{equation}

It is straightforward that the minimization concerning $z$ is equivalent to minimizing each entry $z_j$ independently. 

For $j \in \Omega \setminus \Gamma $, an optimal solution for (\ref{z2}) is
\begin{equation}
    z^{*}_j=\operatorname{T}_j.
\end{equation}

As for $j \in \Gamma,$ we first minimize (\ref{z2}) without considering the convex set constraint term $\mathbb{I}_{H_j}\left( \cdot \right)$. To be more explicit, we consider a new problem with respect to $z_j$
\begin{equation}
 \begin{aligned}
 z^*_j=\underset{z_j \in \mathbb{C}}{\arg \min }&\left(g_j-\sqrt{|z_j|^{2}+\delta_j}\right)^{2}+ \operatorname{coef}_j\cdot\left|z_j-\operatorname{T}_j\right|^{2}.
 \end{aligned} \\
 \label{z_i sub equation}
 \end{equation}
 
In what follows, we denote the objective function in (\ref{z_i sub equation})  as $E_z$. Since $z_j=|z_j|\cdot \operatorname{sign}(z_j)$, we minimize (\ref{z_i sub equation}) with respect to $|z_j|$ and $\operatorname{sign}(z_j)$ respectively (where $\operatorname{sign}(z_j)=\frac{z_j}{|z_j|}$ if $z_j\neq 0$; otherwise $\operatorname{sign}(0)=\theta$ with an arbitrary constant $\theta \in \mathbb{C}$ with unit length). We can easily obtain 
\begin{equation}
    \operatorname{sign}\left(z^{*}_j\right)=\operatorname{sign} \left(\operatorname{T}_j\right)
\end{equation}
by the same argument as before.

To minimize the relaxed subproblem (\ref{z_i sub equation}) with respect to $|z_j|$, we consider
 \begin{equation}
 \begin{aligned}
 \rho^*=\underset{\rho \in \mathbb{R}^{+}}{\arg \min }&\left(g_j-\sqrt{\rho^{2}+\delta_j}\right)^{2}+ \operatorname{coef}_j\cdot\left(\rho-\left| \operatorname{T}_j\right|\right)^{2}.
 \end{aligned} 
 \label{rho-equation}
 \end{equation}

By the first order optimality condition to (\ref{rho-equation}), we can get: 
\begin{equation}
a (\rho^*)^{4}+b (\rho^*)^{3}+c (\rho^*)^2+d\rho^*+e=0
\label{optimality condition z}
\end{equation}
with
\begin{equation*}
\begin{aligned}
a&=\left( \operatorname{coef}_j+1 \right)^2 ,\\
b&=-2\operatorname{coef}_j\cdot\left( \operatorname{coef}_j+1 \right)\left| \operatorname{T}_j\right|,\\
c&=\delta_j\left( \operatorname{coef}_j+1 \right)^2+ \operatorname{coef}_j^2\cdot\left|  \operatorname{T}_j\right|^{2}-g_j^2 ,\\
d&=-2\delta_j \cdot \operatorname{coef}_j\cdot\left( \operatorname{coef}_j+1 \right)\left| \operatorname{T}_j\right|, \\
e&=\left( \operatorname{coef}_j+1 \right)^2\left|  \operatorname{T}_j\right|^2.
\end{aligned}
\end{equation*}

There are four roots for (\ref{optimality condition z}), and we should choose the real non-negative roots. Unfortunately, there may be no real non-negative roots, or there may be multiple minimum points. In these cases, we will take $\rho^*=0$ or the smallest $\rho^*$ as the final solution, respectively.

Recall that the objective function $E_z\left(z_j\right)$ in (\ref{z_i sub equation}) does not contain the indicator term $\mathbb{I}_{H_j} \left( z_j\right)$. If the optimal solution $z_j^*$ happends to be in $H_j$, then we obtain an optimal solution to (\ref{z2}). If this is not the case, the solution to (\ref{z_i sub equation}) is then used as an initial point of a projected gradient descent method for solving (\ref{z2}). To summarize, the final optimal solution of (\ref{z2}) is given by
\begin{equation}
\begin{aligned}
z^{k+1}_j= \left\{\begin{array}{cc}
z^*_j , &  j \in  \Gamma , \\
\operatorname{sign}\left(\operatorname{T}_j\right)\cdot \rho^*_j , &  z^*_j \in H_j \ \text{and} \  j \in \Omega\backslash \Gamma  ,\\
 \operatorname{PG}_{H_j}\left[ E_z(z_j) \right] , & z^*_j \notin H_j \ \text{and} \  j \in \Omega\backslash \Gamma,
\end{array}\right.
\end{aligned}
\label{zregion2}
\end{equation}
where $\rho^*_j$ denotes the final solution of (\ref{rho-equation}) for index $j$, and $\operatorname{PG}_{H_j}\left[ E_z(z_j) \right]$ is the output of projected gradient algorithm for minimizing $E_z(z_j)$ over $H_j$.

\subsection{p-subproblem}  The $p$-subproblem is equivalent to
\begin{equation}
\begin{aligned}
p^{k+1}=\underset{p \in \mathbb{C}^{2n}}{\arg \min } \ &\lambda\|p\|_{1}
+\operatorname{Re}(\langle p-\nabla u^{k+1}, q^k\rangle)+\frac{\gamma}{2}\|p-\nabla u^{k+1}\|^{2}+\left\| p-p^k \right\|_{S_3}^2.
\end{aligned}
\label{p equation 1}
\end{equation}

We simply consider a diagonal proximal matrix $S_3$, and for each index $j$ we have
\begin{equation}
\begin{aligned}
    p^{k+1}(j)&=\underset{p_j \in \mathbb{C}}{\arg \min } \ \lambda|p_j|
+\left(\frac{\gamma}{2}+(S_3)_{j}\right)\left|p_j-\frac{\gamma \left(\nabla u^{k+1}\right)_j-2 q^{k}_j+2 (S_3)_{j} p^{k}_j}{\gamma+2 (S_3)_{j}}\right|^{2},
\end{aligned}
\end{equation}
where $(S_3)_{j}$ is the $j^{th}$ diagonal element of $S_3$.

The minimizer is 
\begin{equation}
\begin{aligned}
    p^{k+1}_j&=\mathcal{T}_{\frac{\lambda}{\gamma+2(S_3)_j}}\left(  \frac{\gamma \left(\nabla u^{k+1}\right)_j-2 q^{k}_j+2 (S_3)_j p^{k}_j}{\gamma+2 (S_3)_j}  \right),
    \end{aligned}
\end{equation}
where $\mathcal{T}_{\lambda}(x)=\operatorname{sign}(x)(|x|-\lambda)_{+} $ with $x\in \mathbb{C}$ is the soft thresholding operator.

The convergence result of the \Cref{alg:MpADMMcam} is given as follows. Please refer to \ref{app:proof} for the proof.

\begin{theorem}\label{thm:convergence}
Let ${(u^k,z^k,p^k,w^k,q^k)}$ be the sequence generated by sPADMM. Suppose $S_1$ is positive definite. If $\tau\in(0,(1+\sqrt{5})/2)$, then ${(u^k,z^k,p^k,w^k,q^k)}$ converges to an optimal point of (\ref{model:relaxed discrete model}).
\end{theorem}

\section{Experimental results}
\label{sec:experiments}
In this section, we focus on Fourier measurements which is common in PR problems. In fact, we consider the specific linear operator $A$, CDP with random masks. For coded diffraction patterns, all the elements of $D_j$  are randomly chosen from $\{\pm \sqrt{2} / 2, \pm \sqrt{2} \mathbf{i} / 2, \pm \sqrt{3}, \pm \sqrt{3} \mathbf{i}\}$ in our experiments.

The real-valued 'Cameraman', 'Livingroom', 'Peppers', 'Pirate' and 'Woman' images with resolution $256 \times 256$ are used as the testing images. 
The testing images are shown in \Cref{testing fig}. Besides, we generate noisy observation measurements according to (\ref{noisy model}) where $\xi$ is additive i.i.d Gaussian noise with mean $0$ and noise level $\sigma$.

\begin{figure}[htbp]
\centering
\begin{minipage}[t]{0.19\linewidth}
\centering
\includegraphics[width=\linewidth]{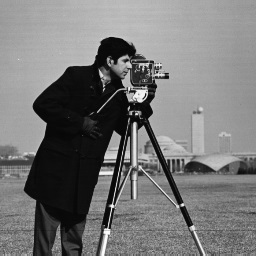}
\centerline{\scriptsize (a)}
\end{minipage}
\begin{minipage}[t]{0.19\linewidth}
\centering
\includegraphics[width=\linewidth]{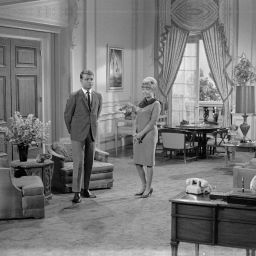}
\centerline{\scriptsize (b)}
\end{minipage}
\begin{minipage}[t]{0.19\linewidth}
\centering
\includegraphics[width=\linewidth]{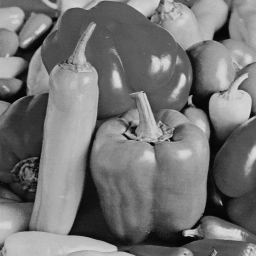}
\centerline{\scriptsize (c)}
\end{minipage}
\begin{minipage}[t]{0.19\linewidth}
\centering
\includegraphics[width=\linewidth]{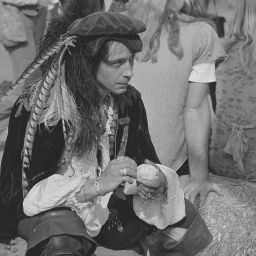}
\centerline{\scriptsize (d)}
\end{minipage}
\begin{minipage}[t]{0.19\linewidth}
\centering
\includegraphics[width=\linewidth]{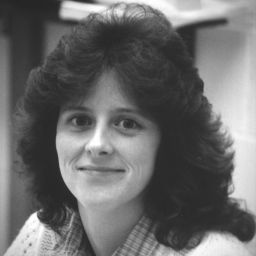}
\centerline{\scriptsize (e)}
\end{minipage}

\centering
\caption{The testing images. The real-valued images: $(a)$ ‘Cameraman’ with resolution $256\times 256$, $(b)$ ‘Livingroom’ with resolution $256 \times 256$, $(c)$ ‘Peppers’ with resolution $256 \times 256$, $(d)$ ‘Pirate’ with  resolution $256 \times 256$ and $(e)$ ‘Woman’ with resolution $256 \times 256$.}
\label{testing fig}
\end{figure}

To measure the reconstruction quality, we use the signal-noise-ratio (SNR), which is defined as
$$
\operatorname{SNR}\left(u, f\right)=-10 \log _{10} \frac{\sum_{j=1}^n\left|u_j-c^{*} f_j\right|^{2}}{\sum_{j=1}^n|u_j|^{2}},
$$
where $f$ is the ground-truth image of size $n\times 1$, $u$ is the reconstructed image, and $c^{*}$ denotes the global phase factor as $c^{*}=\arg \min _{\{c \in \mathbb{C}:|c|=1\}}\left\|u-c f\right\| .$ 

\subsection{Comparison to other PR methods}
In this experiment, we compare our proposed model with five PR methods under noisy CDP measurements: Wirtinger flow (WF) \cite{candes2015phase}, truncated Wirtinger flow (TWF) \cite{chen2017solving}, error reduction (ER) \cite{gerchberg1972practical}, relaxed averaged alternating reflection (RAAR) \cite{wang2017solving}, and TVB \cite{chang2018total}. The codes of TVB are provided by the corresponding author. For TWF and WF methods, we implement the code Phasepack-Matlab\footnote{Available at https://github.com/tomgoldstein/phasepack-matlab.}. As for the remaining three methods, we implement them according to \cite{chang2018total}. Since $A^*A$ is invertible and $\operatorname{Im}(A^*A)=0$ in this case, we introduce two projection operators as $\mathcal{P}_{1}(z)=g \circ \operatorname{sign}(z)$ and $\mathcal{P}_{2}(z)=A\left(A^{*} A\right)^{-1} A^{*} z$, with measurements $g$. Then the iterative algorithms for ER and RAAR with relaxed parameters $\phi>0$ are as follows:
\begin{equation}
\left\{\begin{aligned}
\mathrm{ER:} \ & z^{k+1}=\mathcal{P}_{2} \mathcal{P}_{1}\left(z^{k}\right) \\
\text {RAAR:} \  & z^{k+1}=\left(2 \phi \mathcal{P}_{2} \mathcal{P}_{1}+\phi \mathcal{I}-\phi \mathcal{P}_{2}+(1-2 \phi) \mathcal{P}_{1}\right)\left(z^{k}\right)
\end{aligned}\right.
\end{equation}
for $k=0,1,\cdots$. After getting $z^{k+1}$, we finally compute 
\begin{equation}
u^{k+1}=\left(A^{*} A\right)^{-1} A^{*} z^{k+1}
\end{equation}
as the updated $u$. For real-valued image, we use the modified projection operator $\widetilde{\mathcal{P}}_{2}(z)=A \operatorname{Re}(\left(A^{*} A\right)^{-1} A^{*} z)$ instead of $\mathcal{P}_{2}$ and update $u$ as $u^{k+1}=\operatorname{Re}(\left(A^{*} A\right)^{-1} A^{*} z^{k+1})$. In our experiments, we set $\phi=0.85$ for the above iterative algorithms. Since TVB and our method are regularized methods, for a fair comparison, we apply a denoising procedure on the outputs of ER, RAAR, WF and TWF. BM3D \cite{dabov2007image}, which has great visual and numerical results, is used as the denoising method.

We first conduct the real-valued experiment for noisy CDP measurements with $J=2$ and two noise levels, $\sigma=10$ and $ 20$. The SNR values of all the recovered images are reported in \Cref{table:table-snr sigma}. The recovered images of the compared methods with $\sigma=10$ and $\sigma=20$ are shown in \Cref{sig10} and \Cref{sig20} respectively. Since the pixel value of the test image will be converted into the range of 0 to 1, $\|z-A u\|^{2}$ will be significantly larger than $\|p-\nabla u\|^2$ in (\ref{Lmain-equation}). Hence, we choose $\alpha$ around $1\times 10^1$ and $\gamma$ around $ 3\times 10^5$. $\eta$ determines the degree of dependence of the final solution of the model on the initial point. From (\ref{disk_con}), we choose $\eta$ around $1$. In all the experiments, the parameters of the proposed method are chosen to be $\delta_j=1\times 10^{-2}$ for all $j$, $\alpha=3$, $\gamma=5\times 10^5$, and $\eta=1$. The proximal terms are simply chosen as  the identity matrix for $S_1$, $S_2$ and $S_3$. For the first experiment, we choose $\lambda=2\times 10^3$ and  $\lambda=1\times 10^4$ for noise level $\sigma=10$ and $\sigma=20$, respectively.

\begin{table}[]
\centering
\caption{The SNRs of reconstruction image from different PR methods for $\sigma=10$ and $\sigma=20$ with $J=2$.}
\label{table:table-snr sigma}
\resizebox{\linewidth}{!}{%
\begin{tabular}{|c|c|c|c|c|c|c|}
\hline
\multirow{4}{*}{$\sigma=10$} & ER      & RAAR      & WF      & TWF      & TVB            & Ours           \\ \cline{2-7} 
                             & 18.88   & 18.86     & 18.41   & 15.113   & 25.80          & \textbf{26.49} \\ \cline{2-7} 
                             & ER+BM3D & RAAR+BM3D & WF+BM3D & TWF+BM3D & Initialization &                \\ \cline{2-7} 
                             & 25.87   & 25.85     & 25.74   & 22.40    & 25.79          &                \\ \hline
\multirow{4}{*}{$\sigma=20$} & ER      & RAAR      & WF      & TWF      & TVB            & Ours           \\ \cline{2-7} 
                             & 12.60   & 12.45     & 12.79   & 10.70    & 22.48          & \textbf{22.62} \\ \cline{2-7} 
                             & ER+BM3D & RAAR+BM3D & WF+BM3D & TWF+BM3D & Initialization &                \\ \cline{2-7} 
                             & 21.90   & 21.83     & 21.64   & 19.38    & 21.39          &                \\ \hline
\end{tabular}%
}
\end{table}

The proposed method showed improvement in both numerical values and visual results over the compared methods.  We run five different PR methods and the denoised versions of ER, RAAR, WF, and TWF to compare their effectiveness. ER and BM3D will be used as the initial procedures, where we run ER for $40$ iterations, then a rough initialization will be generated based on the output of ER using BM3D. The proposed method's SNR value is at least 7dB higher than ER, RAAR, WF, and TWF when $\sigma=10$. What's more, after using BM3D as the denoising algorithm, our method still outperformed the denoised ones. Similarly, compared to a total-variation-based method TVB, we can achieve almost 0.7dB of improvement. Interestingly, for a higher noise level $\sigma=20$, the SNR value of the proposed method is even drastically higher than ER, RAAR, WF, and TWF. Visually, the compared methods produce very noisy results. With the TV regularization term, the proposed method and TVB are robust to noise and the proposed method can produce recovered images with a visually clearer background, showcasing the proposed method's effectiveness in the presence of noise.


       



\begin{figure}[htbp]
\begin{minipage}[t]{0.24\linewidth}
\centering
\includegraphics[width=.9\linewidth]{images/cameraman.jpg}
\centerline{\scriptsize (a) Original image }
\end{minipage}
\begin{minipage}[t]{0.24\linewidth}
\centering
\includegraphics[width=.9\linewidth]{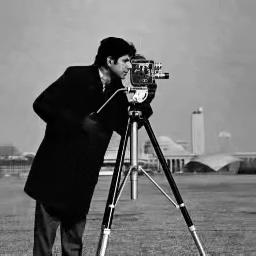}
\centerline{\scriptsize (b) Initilization}
\centerline{\scriptsize SNR: 25.79 }
\end{minipage}
\begin{minipage}[t]{0.24\linewidth}
\centering
\includegraphics[width=.9\linewidth]{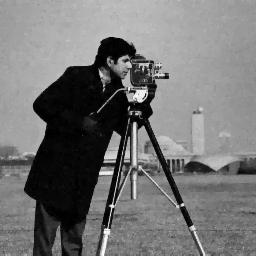}
\centerline{\scriptsize (c) TVB}
\centerline{\scriptsize SNR: 25.80 }
\end{minipage}
\begin{minipage}[t]{0.24\linewidth}
\centering
\includegraphics[width=.9\linewidth]{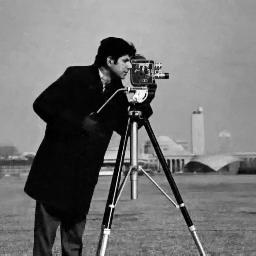}
\centerline{\scriptsize (d) Ours}
\centerline{\scriptsize SNR: \textbf{26.49} }
\end{minipage}

\begin{minipage}[t]{0.24\linewidth}
\centering
\includegraphics[width=.9\linewidth]{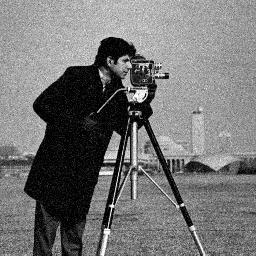}
\centerline{\scriptsize (e) ER}
\centerline{\scriptsize SNR: 18.88 }
\end{minipage}
\begin{minipage}[t]{0.24\linewidth}
\centering
\includegraphics[width=.9\linewidth]{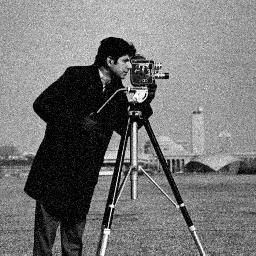}
\centerline{\scriptsize (f) RAAR}
\centerline{\scriptsize SNR: 18.86 }
\end{minipage}
\begin{minipage}[t]{0.24\linewidth}
\centering
\includegraphics[width=.9\linewidth]{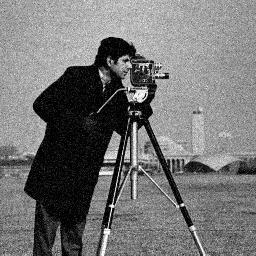}
\centerline{\scriptsize (g) WF}
\centerline{\scriptsize SNR: 18.41}
\end{minipage}
\begin{minipage}[t]{0.24\linewidth}
\centering
\includegraphics[width=.9\linewidth]{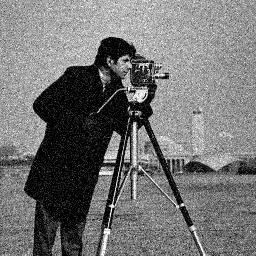}
\centerline{\scriptsize (h) TWF}
\centerline{\scriptsize SNR: 15.13 }
\end{minipage}

\begin{minipage}[t]{0.24\linewidth}
\centering
\includegraphics[width=.9\linewidth]{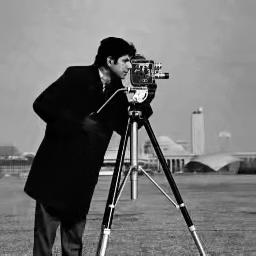}
\centerline{\scriptsize (i) ER+BM3D}
\centerline{\scriptsize SNR: 25.87 }
\end{minipage}
\begin{minipage}[t]{0.24\linewidth}
\centering
\includegraphics[width=.9\linewidth]{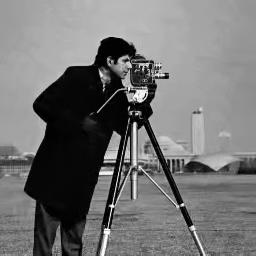}
\centerline{\scriptsize (j) RAAR+BM3D}
\centerline{\scriptsize SNR: 25.85 }
\end{minipage}
\begin{minipage}[t]{0.24\linewidth}
\centering
\includegraphics[width=.9\linewidth]{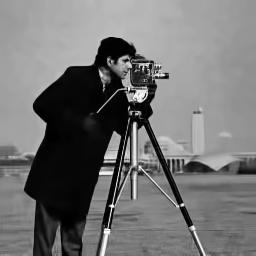}
\centerline{\scriptsize (k) WF+BM3D}
\centerline{\scriptsize SNR: 25.74}
\end{minipage}
\begin{minipage}[t]{0.24\linewidth}
\centering
\includegraphics[width=.9\linewidth]{images/exper1212/TVBL2sigma10.jpg}
\centerline{\scriptsize (l) TWF+BM3D}
\centerline{\scriptsize SNR: 22.40 }
\end{minipage}
\caption{Comparison of PR methods from noisy CDP measurements (\ref{CDP})  with $J=2$ and noise level $ \sigma=10$. }
\label{sig10}
\end{figure}


       


\begin{figure}[htbp]
\begin{minipage}[t]{0.24\linewidth}
\centering
\includegraphics[width=.9\linewidth]{images/cameraman.jpg}
\centerline{ \scriptsize (a) Original image}
\end{minipage}
\begin{minipage}[t]{0.24\linewidth}
\centering
\includegraphics[width=.9\linewidth]{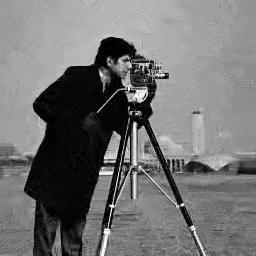}
\centerline{\scriptsize (b) Initilization}
\centerline{\scriptsize SNR: 21.39 }
\end{minipage}
\begin{minipage}[t]{0.24\linewidth}
\centering
\includegraphics[width=.9\linewidth]{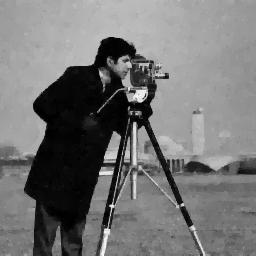}
\centerline{\scriptsize (c) TVB}
\centerline{\scriptsize SNR: 22.48 }
\end{minipage}
\begin{minipage}[t]{0.24\linewidth}
\centering
\includegraphics[width=.9\linewidth]{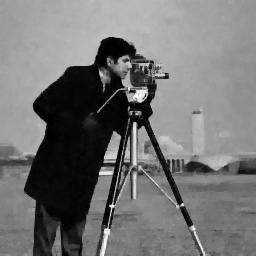}
\centerline{\scriptsize (d) Ours}
\centerline{\scriptsize SNR: \textbf{22.62}}
\end{minipage}

\begin{minipage}[t]{0.24\linewidth}
\centering
\includegraphics[width=.9\linewidth]{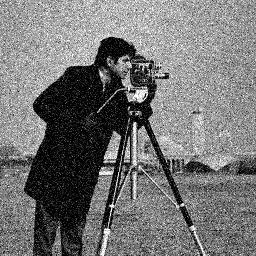}
\centerline{\scriptsize (e) ER}
\centerline{\scriptsize SNR: 12.60 }
\end{minipage}
\begin{minipage}[t]{0.24\linewidth}
\centering
\includegraphics[width=.9\linewidth]{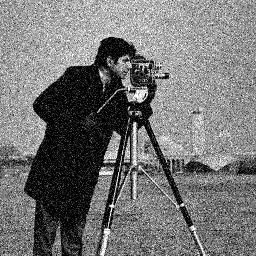}
\centerline{\scriptsize (f) RAAR}
\centerline{\scriptsize SNR: 12.45 }
\end{minipage}
\begin{minipage}[t]{0.24\linewidth}
\centering
\includegraphics[width=.9\linewidth]{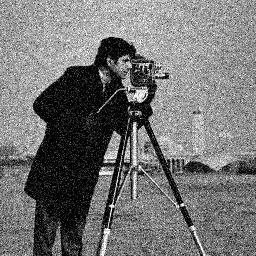}
\centerline{\scriptsize (g) WF}
\centerline{\scriptsize SNR: 12.79 }
\end{minipage}
\begin{minipage}[t]{0.24\linewidth}
\centering
\includegraphics[width=.9\linewidth]{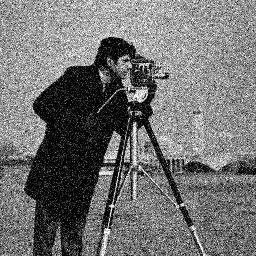}
\centerline{\scriptsize (h) TWF}
\centerline{\scriptsize SNR: 10.70 }
\end{minipage}

\begin{minipage}[t]{0.24\linewidth}
\centering
\includegraphics[width=.9\linewidth]{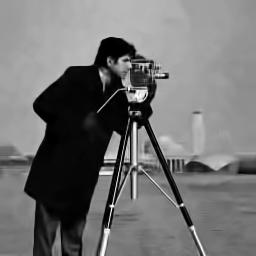}
\centerline{\scriptsize (i) ER+BM3D}
\centerline{\scriptsize SNR: 21.90}
\end{minipage}
\begin{minipage}[t]{0.24\linewidth}
\centering
\includegraphics[width=.9\linewidth]{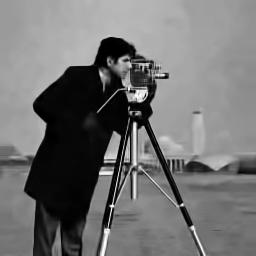}
\centerline{\scriptsize (j) RAAR+BM3D}
\centerline{\scriptsize SNR: 21.83}
\end{minipage}
\begin{minipage}[t]{0.24\linewidth}
\centering
\includegraphics[width=.9\linewidth]{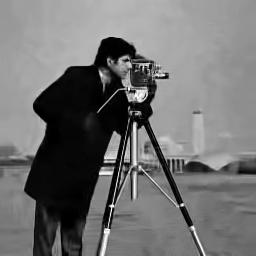}
\centerline{\scriptsize (k) WF+BM3D}
\centerline{\scriptsize SNR: 21.64 }
\end{minipage}
\begin{minipage}[t]{0.24\linewidth}
\centering
\includegraphics[width=.9\linewidth]{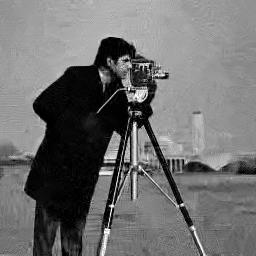}
\centerline{\scriptsize (l) TWF+BM3D}
\centerline{\scriptsize SNR: 19.38 }
\end{minipage}
\caption{Comparison of PR methods from noisy CDP measurements (\ref{CDP})  with $J=2$ and noise level $ \sigma=20$.}
\label{sig20}
\end{figure}

\begin{figure}[htbp]
\begin{minipage}[t]{0.24\linewidth}
\centering
\includegraphics[width=.9\linewidth]{images/cameraman.jpg}
\centerline{\scriptsize (a) Original image}
\end{minipage}
\begin{minipage}[t]{0.24\linewidth}
\centering
\includegraphics[width=.9\linewidth]{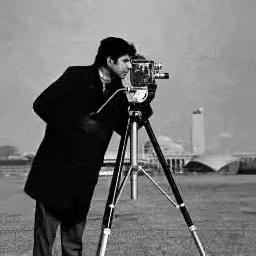}
\centerline{\scriptsize (b) Initilization}
\centerline{\scriptsize SNR: 23.56 }
\end{minipage}
\begin{minipage}[t]{0.24\linewidth}
\centering
\includegraphics[width=.9\linewidth]{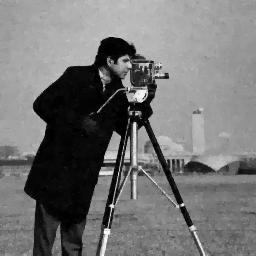}
\centerline{\scriptsize (c) TVB}
\centerline{\scriptsize SNR: 23.81 }
\end{minipage}
\begin{minipage}[t]{0.24\linewidth}
\centering
\includegraphics[width=.9\linewidth]{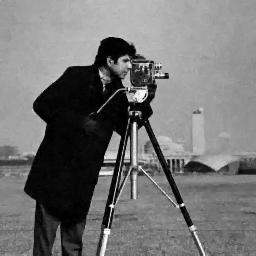}
\centerline{\scriptsize (d) Ours}
\centerline{\scriptsize SNR: \textbf{24.30}}
\end{minipage}

\begin{minipage}[t]{0.24\linewidth}
\centering
\includegraphics[width=.9\linewidth]{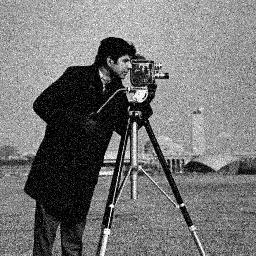}
\centerline{\scriptsize (e) ER}
\centerline{\scriptsize SNR: 15.67 }
\end{minipage}
\begin{minipage}[t]{0.24\linewidth}
\centering
\includegraphics[width=.9\linewidth]{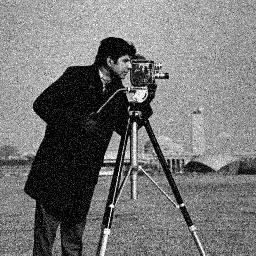}
\centerline{\scriptsize (f) RAAR}
\centerline{\scriptsize SNR: 15.59}
\end{minipage}
\begin{minipage}[t]{0.24\linewidth}
\centering
\includegraphics[width=.9\linewidth]{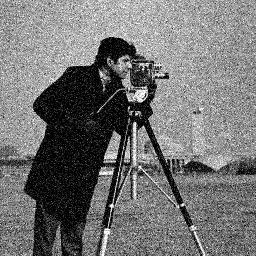}
\centerline{\scriptsize (g) WF}
\centerline{\scriptsize SNR: 13.16 }
\end{minipage}
\begin{minipage}[t]{0.24\linewidth}
\centering
\includegraphics[width=.9\linewidth]{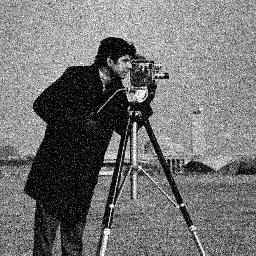}
\centerline{\scriptsize (h) TWF}
\centerline{\scriptsize SNR: 12.54 }
\end{minipage}

\begin{minipage}[t]{0.24\linewidth}
\centering
\includegraphics[width=.9\linewidth]{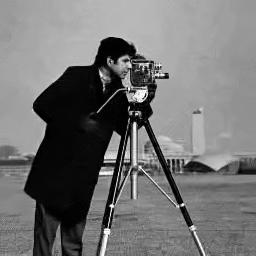}
\centerline{\scriptsize (i) ER+BM3D}
\centerline{\scriptsize SNR: 23.96}
\end{minipage}
\begin{minipage}[t]{0.24\linewidth}
\centering
\includegraphics[width=.9\linewidth]{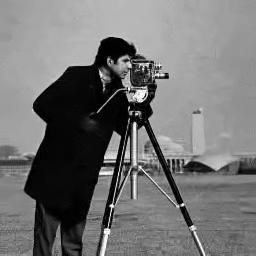}
\centerline{\scriptsize (j) RAAR+BM3D}
\centerline{\scriptsize SNR: 23.90 }
\end{minipage}
\begin{minipage}[t]{0.24\linewidth}
\centering
\includegraphics[width=.9\linewidth]{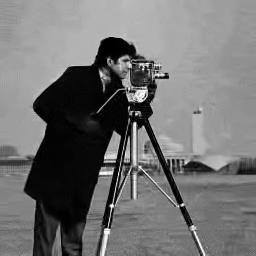}
\centerline{\scriptsize (k) WF+BM3D}
\centerline{\scriptsize SNR: 23.16}
\end{minipage}
\begin{minipage}[t]{0.24\linewidth}
\centering
\includegraphics[width=.9\linewidth]{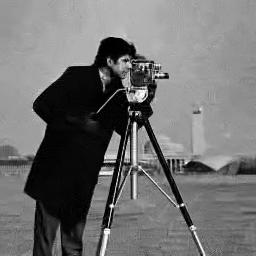}
\centerline{\scriptsize (l) TWF+BM3D}
\centerline{\scriptsize SNR: 21.40}
\end{minipage}
\caption{Comparison of PR methods from noisy CDP measurements (\ref{CDP})  with $J=3$ and noise level $ \sigma=20$.}
\label{fig:J=3}
\end{figure}

\begin{table}[]
\centering
\caption{The SNRs of reconstruction image from different PR methods for $J=2$ and $J=3$ with $\sigma=20$.}
\label{table:table-masks}
\resizebox{\linewidth}{!}{%
\begin{tabular}{|c|c|c|c|c|c|c|}
\hline
\multirow{4}{*}{$J=2$} & ER & RAAR & WF & TWF & TVB & Ours \\ \cline{2-7} 
 & 12.60 & 12.45 & 12.79 & 10.70 & 22.48 & \textbf{22.62} \\ \cline{2-7} 
 & ER+BM3D & RAAR+BM3D & WF+BM3D & TWF+BM3D & Initialization &  \\ \cline{2-7} 
 & 21.90 & 21.83 & 21.64 & 19.38 & 21.39 &  \\ \hline
\multirow{4}{*}{$J=3$} & ER & RAAR & WF & TWF & TVB & Ours \\ \cline{2-7} 
 & 15.67 & 15.59 & 13.16 & 12.54 & 23.81 & \textbf{24.30} \\ \cline{2-7} 
 & ER+BM3D & RAAR+BM3D & WF+BM3D & TWF+BM3D & Initialization &  \\ \cline{2-7} 
 & 23.96 & 23.90 & 23.16 & 21.40 & 23.56 &  \\ \hline
\end{tabular}%
}
\end{table}

We then study the effect of the number of masks. With the number of masks being $J=2$ and $J=3$, we choose $\lambda=1\times 10^4$ and $\lambda=7\times 10^3$ respectively for the proposed method. Similarly, we run five different PR methods and the denoised versions of ER, RAAR, WF, and TWF to compare their effectiveness. ER and BM3D are also used as the initial procedures. \Cref{table:table-masks} shows that $3n$ measurements have better reconstruction quality than that of $2n$ measurements for all the methods. Our method is almost 10dB higher in SNR values than ER, RAAR, WF, and TWF. Nonetheless, when the measurements increase to $3n$, our method is at least 8dB higher in SNR than theirs. For the denoised versions of ER, RAAR, WF, TWF, and the regularized method TVB, the improvement is still obvious. Visual results of $3n$ measurements are shown in \Cref{fig:J=3}. We also plot the error curves of the compared methods in \Cref{fig:error-curve}.

\begin{figure}[htbp]
\centering

\begin{minipage}[t]{0.45\linewidth}
\centering
\includegraphics[width=\linewidth]{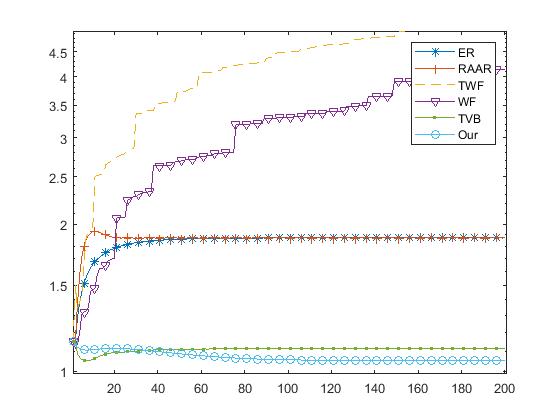}
\centerline{\scriptsize  (c) $L=2$, $\sigma=10$}
\end{minipage}
\begin{minipage}[t]{0.45\linewidth}
\centering
\includegraphics[width=\linewidth]{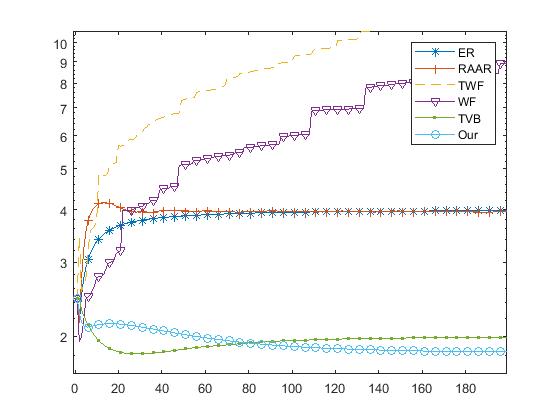}
\centerline{\scriptsize  (c) $L=2$, $\sigma=20$}
\end{minipage}
\begin{minipage}[t]{0.45\linewidth}
\centering
\includegraphics[width=\linewidth]{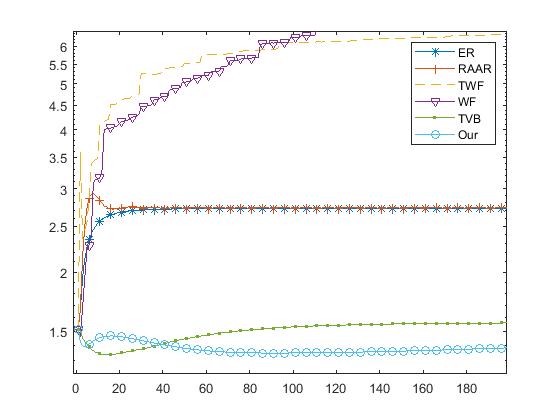}
\centerline{\scriptsize  (c) $L=3$, $\sigma=20$}
\end{minipage}
\centering
\caption{ The error curves for the proposed method with different noise levels and masks.  }
\label{fig:error-curve}
\end{figure}

\subsection{Effectiveness of TV regularization}
In this experiment, we study the effectiveness of TV regularization. We take $\lambda=0$ and $\lambda=8\times 10^3$ in the proposed method as the non-TV algorithm and the TV algorithm, respectively. Besides, CDP measurements with $J=2$ are used in this experiment, and we further add Gaussian noise with a noise level of $\sigma=20$ to the corresponding measurements. The ER algorithm and BM3D are used as the initial procedures for both the non-TV and TV algorithms. The other parameters remain the same as in the real image experiment.   

After the initial procedure, we run the non-TV and TV algorithms for $60$ iterations. The reconstructed results are shown in \Cref{nonTV}. With the TV regularization, the TV algorithm shows a huge improvement in SNR value over the non-TV algorithm. For the reconstructed images by the non-TV algorithm, prominent noise can still be observed, while the algorithm with TV successfully removes the noise. This demonstrates the importance of TV regularization to the phase retrieval problem in the presence of noise. 

\begin{figure}[htbp]
\centering

\begin{minipage}[t]{0.32\linewidth}
\centering
\includegraphics[width=1.2in]{images/livingroom.jpg}
\centerline{ \scriptsize  (a) Livingroom }
\end{minipage}
\begin{minipage}[t]{0.32\linewidth}
\centering
\includegraphics[width=1.2in]{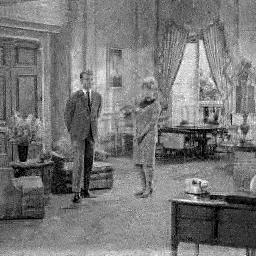}
\centerline{ \scriptsize  (b) SNR=18.33 }
\end{minipage}
\begin{minipage}[t]{0.32\linewidth}
\centering
\includegraphics[width=1.2in]{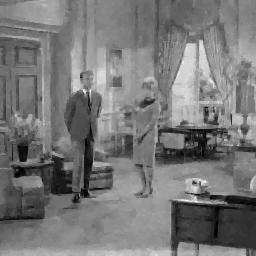}
\centerline{ \scriptsize  (c) SNR=21.22 }
\end{minipage}

\quad

\begin{minipage}[t]{0.32\linewidth}
\centering
\includegraphics[width=1.2in]{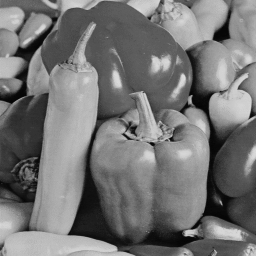}
\centerline{ \scriptsize (d) Peppers }
\end{minipage}
\begin{minipage}[t]{0.32\linewidth}
\centering
\includegraphics[width=1.2in]{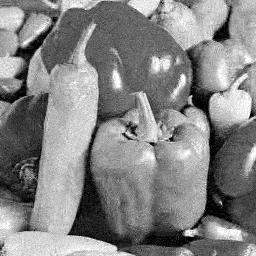}
\centerline{ \scriptsize  (e) SNR=20.89 }
\end{minipage}
\begin{minipage}[t]{0.32\linewidth}
\centering
\includegraphics[width=1.2in]{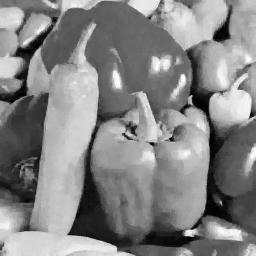}
\centerline{ \scriptsize  (f) SNR=24.28 }
\end{minipage}

\quad

\begin{minipage}[t]{0.32\linewidth}
\centering
\includegraphics[width=1.2in]{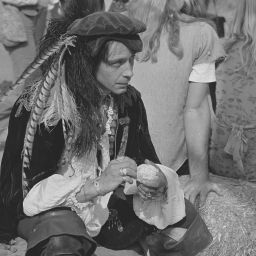}
\centerline{ \scriptsize (g) Pirate }
\end{minipage}
\begin{minipage}[t]{0.32\linewidth}
\centering
\includegraphics[width=1.2in]{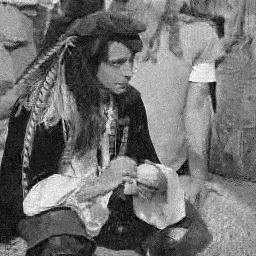}
\centerline{ \scriptsize  (h) SNR=18.73 }
\end{minipage}
\begin{minipage}[t]{0.32\linewidth}
\centering
\includegraphics[width=1.2in]{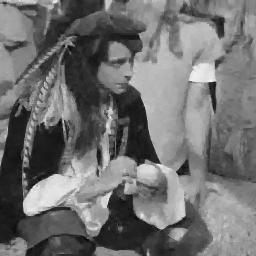}
\centerline{ \scriptsize  (i) SNR=20.64 }
\end{minipage}

\quad

\begin{minipage}[t]{0.32\linewidth}
\centering
\includegraphics[width=1.2in]{images/woman_darkhair.jpg}
\centerline{ \scriptsize (j) Woman }
\end{minipage}
\begin{minipage}[t]{0.32\linewidth}
\centering
\includegraphics[width=1.2in]{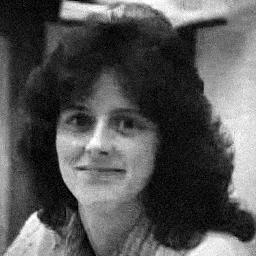}
\centerline{ \scriptsize  (k) SNR=20.26}
\end{minipage}
\begin{minipage}[t]{0.32\linewidth}
\centering
\includegraphics[width=1.2in]{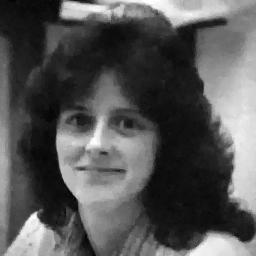}
\centerline{ \scriptsize  (l) SNR=23.86 }
\end{minipage}

\centering
\caption{ Comparison of non-TV and TV methods from CDP measurements with $J=2$ and noise level $ \sigma=20$. The first column contains four original $256 \times 256$ images. The second and third columns show the reconstructed images for the non-TV model and the TV model.   }
\label{nonTV}
\end{figure}

\section{Conclusions}
\label{sec:conclusions}
In this paper, we proposed a convex augmentation phase retrieval model based on total variation regularization. With the TV regularization, the proposed model can handle incomplete and noisy magnitude measurements. By incorporating different regularizers with the proposed convex augmentation technique, we can further improve the phase retrieval model. Furthermore, unlike convex relaxations, the proposed model is scalable and can be efficiently solved by the modified sPADMM algorithm with guaranteed convergence. The modified sPADMM is more flexible and can be applied to other problems with multiple linear constraints. The modified algorithm and the convergence analysis may be inspirational to future works. Numerically, our proposed model can recover the images with a high level of noise and preserve sharp edges at the same time. The numerical results also showcase the excellent performance of the proposed method.

For future works, we aim to extend the convex augmentation model to phase retrieval problems with multiplicative noise or Rician noise. For example. extending the convex variation model of \cite{dong2013convex} is worth studying. Moreover, smoothing methods \cite{chen2010smoothing} have been introduced to solve non-smooth and non-convex problems. We are interested in incorporating such smoothing techniques in the future with our TV-based method for a better model. A more accurate initialization to improve the numerical results will also be considered. In this paper, the simple ER method was used to compute the initial point $\hat{u}$. Due to the $\hat{z}$ term in the model, the result of the proposed model depends on the initialization. Choosing a method that takes into account the convex structure of the proposed method will be considered in the future.

\section*{Acknowledgments}
The work of Tieyong Zeng was supported by the CUHK startup, and the CUHK DAG under grant 4053342, 4053405, RGC 14300219, RGC 14302920, and grant NSFC/RGC N{\_}CUHK 415/19.

\begin{appendix}
\section{Proof of \Cref{thm:convergence}}
\label{app:proof}
We include the proof of convergence of \Cref{alg:MpADMMcam}. The proof follows \cite{fazel2013hankel}.

\textbf{Proof:} Let
\begin{equation*}
    \begin{aligned}
        F(p):= &\lambda\|p\|_1,\\
        G(z):= &\sum_{j\in \Gamma}(g_j-\sqrt{|z_j|^2+\delta})^2
    +\eta\sum_{j\in \Omega}|z_j-\hat{z}_j|^2+\sum_{j\in \Gamma}\mathbb{I}_{H_j(\hat{z}_j)}(z_j).
    \end{aligned}
\end{equation*}

In the following, we consider a complex variable $\mu\in\mathbb{C}^K$ as a real variable in $\mathbb{R}^{2K}$. The inner product $\langle \cdot,\cdot \rangle$ is also used to denote the product $\operatorname{Re}(\langle \cdot,\cdot\rangle)$. Similarly, $F$ and $G$ are considered as functions in real variables.

Since (\ref{model:relaxed discrete model}) is a convex problem with affine constraints, then $(u^*,z^*,p^*)$ is an optimal solution of (\ref{model:relaxed discrete model}) if there exists Lagrange multipliers $w^*,q^*$ such that 
\begin{equation}
\begin{aligned}
    &-A^*w^*-\nabla^Tq^*=0\\
    &-w^*\in\partial G(z^*)\\
    &-q^*\in\partial F(p^*)\\
    &z^*-Au^*=0,\ p^*-\nabla u^*=0.
\end{aligned}
\label{kkt-condition}
\end{equation}
Since the subdifferential mappings of closed convex functions are maximal monotone \cite{rockafellar2009variational}, there exists self-adjoint and positive semidefinite operators $\Sigma_F,\Sigma_G$ such that for all $p_1,p_2\in\text{dom}(F),\ x_1\in\partial F(p_1)$ and $x_2\in\partial F(p_2)$, 
\begin{equation}
    \begin{aligned}
        &F(p_1)\geq F(p_2)+\langle x_2,p_1-p_2\rangle+\frac{1}{2}\|p_1-p_2\|^2_{\Sigma_F},\\
        &\langle x_1-x_2,p_1-p_2\rangle\geq\|p_1-p_2\|^2_{\Sigma_F},
    \end{aligned}
\label{subgrad-F}
\end{equation}
and for all $z_1,z_2\in\text{dom}(G),\ y_1\in\partial G(z_1)$ and $y_2\in\partial F(z_2)$, 
\begin{equation}
    \begin{aligned}
        &G(z_1)\geq G(z_2)+\langle y_2,z_1-z_2\rangle+\frac{1}{2}\|z_1-z_2\|^2_{\Sigma_G},\\ 
        &\langle y_1-y_2,z_1-z_2\rangle\geq\|z_1-z_2\|^2_{\Sigma_G}.        
    \end{aligned}
\label{subgrad-G}
\end{equation}

The sequence $\{(u^k,z^k,p^k,w^k,q^k)\}$ generated by sPADMM satisfies
\begin{equation}
\begin{aligned}
    0&= -A^*w^k-\nabla^T q^k+\alpha A^*(Au^{k+1}-z^k)\\
    &\ +\gamma\nabla^T(\nabla u^{k+1}-p^k)+S_1(u^{k+1}-u^k)\\
    0&\in\partial G(z^{k+1})+w^k+\alpha(z^{k+1}-Au^{k+1})+S_2(z^{k+1}-z^k)\\
    0&\in\partial F(p^{k+1})+q^k+\gamma(p^{k+1}-\nabla u^{k+1})+S_3(p^{k+1}-p^k)\\
    0&= \epsilon_1(z^{k+1},u^{k+1})-(\tau\alpha)^{-1}(w^{k+1}-w^k)\\
    0&= \epsilon_2(p^{k+1},u^{k+1})-(\tau\gamma)^{-1}(q^{k+1}-q^k),
\end{aligned}
\label{spadmm-condition}
\end{equation}
where $\epsilon_1(z,u)=z-Au,\ \epsilon_2(p,u)=p-\nabla u$. We further define $u^k_e = u^k-u^*$ and similarly for $z^k_e,p^k_e,w^k_e,q^k_e$.\\
Note that
\begin{equation}
\begin{aligned}
    w^{k}&=-w^{k+1}+\tau\alpha\epsilon_1(z^{k+1},u^{k+1}),\\
    q^{k}&=-q^{k+1}+\tau\gamma\epsilon_2(p^{k+1},u^{k+1}).
\end{aligned}
\label{dual-kk1}
\end{equation}

Combining (\ref{subgrad-F}-\ref{dual-kk1}) with (\ref{kkt-condition}), we have
\begin{equation}
\begin{aligned}
    &0=\langle Au^{k+1}_e,w^{k+1}_e-\tau\alpha\epsilon_1(z^{k+1},u^{k+1})\rangle\\&+\langle\nabla u^{k+1}_e,q^{k+1}_e-\tau\gamma\epsilon_2(q^{k+1},u^{k+1})\rangle-\langle u^{k+1}_e,S_1(u^{k+1}-u^k)\rangle,\\
    &\|z^{k+1}_e\|^2_{\Sigma_F}\\&\leq\langle z^{k+1}_e,-w^{k+1}_e-(1-\tau)\alpha\epsilon_1(z^{k+1}_e,w^{k+1}_e)-S_2(z^{k+1}-z^k)\rangle,\\
    &\|p^{k+1}_e\|^2_{\Sigma_G}\\&\leq\langle p^{k+1}_e,-q^{k+1}_e-(1-\tau)\gamma\epsilon_2(p^{k+1}_e,q^{k+1}_e)-S_3(p^{k+1}-p^k)\rangle.
\end{aligned}
\label{main-inequ1}
\end{equation}
Adding the above inequalities, we have
\begin{equation}
    \begin{aligned}
        &\|z^{k+1}_e\|_{\Sigma_F}+\|p^{k+1}_e\|_{\Sigma_G}\leq\\&\langle w^{k+1}_e,Au^{k+1}_e-z^{k+1}_e\rangle+\langle q^{k+1}_e,\nabla u^{k+1}_e-p^{k+1}_e\rangle\\
        &+(1-\tau)\alpha\langle\epsilon_1(z^{k+1},u^{k+1}),Au^{k+1}_e-z^{k+1}_e\rangle\\&+(1-\tau)\gamma\langle\epsilon_2(p^{k+1},u^{k+1}),\nabla u^{k+1}_e-p^{k+1}_e\rangle\\
        &-\alpha\langle z^{k+1}-z^k,Au^{k+1}_e-z^{k+1}_e\rangle-\alpha\langle z^{k+1}-z^k,z^{k+1}_e\rangle\\
        &-\gamma\langle p^{k+1}-p^k,\nabla u^{k+1}_e-p^{k+1}_e\rangle\\&-\gamma\langle p^{k+1}-p^k, \nabla u^{k+1}_e-p^{k+1}_e\rangle\\
        &-\langle u^{k+1}_e,S_1(u^{k+1}-u^k\rangle-\langle z^{k+1}_e,S_2(z^{k+1}-z^k\rangle\\&-\langle p^{k+1}_e,S_3(p^{k+1}-p^k\rangle\\
        &=(\tau\alpha)^{-1}\langle w^{k+1}_e,w^k_e-w^{k+1}_e\rangle+(\tau\gamma)^{-1}\langle q^{k+1}_e,q^k_e-q^{k+1}_e\rangle\\
        &-(1-\tau)\alpha\|\epsilon_1(z^{k+1},u^{k+1}\|^2-(1-\tau)\gamma\|\epsilon_2(p^{k+1},u^{k+1})\|^2\\
        &+\alpha\langle z^{k+1}-Au^{k+1},z^{k+1}-z^k\rangle-\alpha\langle z^{k+1}_e,z^{k+1}-z^k\rangle\\
        &+\gamma\langle p^{k+1}-\nabla u^{k+1},p^{k+1}-p^k\rangle-\gamma\langle p^{k+1}_e,p^{k+1}-p^k\rangle\\
        &-\langle u^{k+1}_e,S_1(u^{k+1}-u^k\rangle-\langle z^{k+1}_e,S_2(z^{k+1}-z^k\rangle\\&-\langle p^{k+1}_e,S_3(p^{k+1}-p^k\rangle.
    \end{aligned}
\label{main-inequ2}
\end{equation}
Now we estimate the terms $\alpha\langle z^{k+1}-Au^{k+1},z^{k+1}-z^k\rangle$ and $\gamma\langle p^{k+1}-\nabla u^{k+1},p^{k+1}-p^k\rangle$.
We have
\begin{equation}
    \begin{aligned}
        &\alpha\langle z^{k+1}-Au^{k+1},z^{k+1}-z^k\rangle\\
        &=(1-\tau)\alpha\langle\epsilon_1(z^{k+1},u^{k+1}),z^{k+1}-z^k\rangle\\&\ +\tau\alpha\langle\epsilon_1(z^{k+1},u^{k+1},z^{k+1}-z^k\rangle\\
        &=(1-\tau)\alpha\langle\epsilon_1(z^{k+1},u^{k+1}),z^{k+1}-z^k\rangle\\&\ +\tau\alpha\langle w^{k+1}-w^k,z^{k+1}-z^k\rangle,\\
        &\gamma\langle p^{k+1}-\nabla u^{k+1},p^{k+1}-p^k\rangle\\
        &=(1-\tau)\gamma\langle\epsilon_2(p^{k+1},u^{k+1}),p^{k+1}-p^k\rangle\\&\ +\tau\gamma\langle\epsilon_2(p^{k+1},u^{k+1},p^{k+1}-p^k\rangle\\
        &=(1-\tau)\gamma\langle\epsilon_2(p^{k+1},u^{k+1}),p^{k+1}-p^k\rangle\\&\ +\tau\gamma\langle q^{k+1}-q^k,p^{k+1}-p^k\rangle.
    \end{aligned}
\label{est1}
\end{equation}
By the sPADMM condition (\ref{spadmm-condition}),
\begin{equation}
    \begin{aligned}
        &-w^{k+1}-(1-\tau)\alpha\epsilon_1(z^{k+1},u^{k+1})-S_2(z^{k+1}-z^k) \\ 
        & \ \in\partial F(z^{k+1}),\\
        &-w^k-(1-\tau)\alpha\epsilon_1(z^k,u^k)-S_2(z^k-z^{k-1})\in\partial F(z^k),\\
        &-q^{k+1}-(1-\tau)\gamma\epsilon_2(p^{k+1},u^{k+1})-S_3(p^{k+1}-p^k)\\ 
        & \ \in\partial G(p^{k+1}),\\
        &-q^k-(1-\tau)\gamma\epsilon_2(p^k,u^k)-S_3(p^k-p^{k-1})\in\partial G(p^k).
    \end{aligned}
\label{est2}
\end{equation}
By the monotonicity of $\partial F,\partial G$ and (\ref{est2}), we have
\begin{equation}
    \begin{aligned}
        -&\langle w^{k+1}-w^k-(1-\tau)\alpha[\epsilon_1(z^{k+1},u^{k+1})-\epsilon_1(z^k,u^k)],\\
        & z^{k+1}-z^k\rangle
        \geq\|z^{k+1}-z^k\|^2_{S_2}-\langle S_2(z^{k+1}-z^k),z^{k+1}-z^k\rangle,\\
        -&\langle q^{k+1}-q^k-(1-\tau)\gamma[\epsilon_2(p^{k+1},u^{k+1})-\epsilon_2(p^k,u^k)],\\
        & p^{k+1}-p^k\rangle
        \geq\|p^{k+1}-p^k\|^2_{S_3}-\langle S_3(p^{k+1}-p^k),p^{k+1}-p^k\rangle.
    \end{aligned}
\label{est3}
\end{equation}
By (\ref{est1}) and (\ref{est3}), then 
\begin{equation}
    \begin{aligned}
        &\alpha\langle\epsilon_1(z^{k+1},u^{k+1}),z^{k+1}-z^k\rangle\\
        &=(1-\tau)\alpha\langle\epsilon_1(z^{k+1},u^{k+1},z^{k+1}-z^k\rangle\\
        &\ +\langle w^{k+1}-w^k,z^{k+1}-z^k\rangle\\
        &\leq(1-\tau)\alpha\langle\epsilon_1(z^k,u^k),z^{k+1}-z^k\rangle-\|z^{k+1}-z^k\|^2_{S_2}\\
        &\ +\langle S_2(z^k-z^{k-1}),z^{k+1}-z^k\rangle\\
        &\leq(1-\tau)\alpha\langle\epsilon_1(z^k,u^k),z^{k+1}-z^k\rangle-\frac{1}{2}\|z^{k+1}-z^k\|^2_{S_2}\\
        & +\frac{1}{2}\|z^k-z^{k-1}\|^2_{S_2},    
    \end{aligned}
\label{estz}
\end{equation}
and 
\begin{equation}
    \begin{aligned}    
        &\gamma\langle\epsilon_2(p^{k+1},u^{k+1}),p^{k+1}-p^k\rangle\\
        &\leq(1-\tau)\gamma\langle\epsilon_2(p^k,u^k),p^{k+1}-p^k\rangle-\frac{1}{2}\|p^{k+1}-p^k\|^2_{S_3}\\
        &+\frac{1}{2}\|p^k-p^{k-1}\|^2_{S_3}.
    \end{aligned}    
\label{estp}
\end{equation}
Define $a_{k+1}=(1-\tau)\alpha\langle\epsilon_1(z^{k+1},u^{k+1},z^{k+1}-z^k\rangle$ and $b_{k+1}=(1-\tau)\gamma\langle\epsilon_2(p^k,u^k),p^{k+1}-p^k\rangle$.
We get from (\ref{main-inequ2}), (\ref{estz}) and (\ref{estp}) that
\begin{equation}
    \begin{aligned}
        &2\|z^{k+1}_e\|^2_{\Sigma_F}+2\|p^{k+1}_e\|^2_{\Sigma_G}\\
        &\leq(\tau\alpha)^{-1}(\|w^k_e\|^2-\|w^{k+1}_e\|^2)+(\tau\gamma)^{-1}(\|p^k_e\|^2-\|p^{k+1}_e\|^2)\\
        &\ -(2-\tau)\alpha\|\epsilon_1(z^{k+1},u^{k+1})\|^2-(2-\tau)\gamma\|\epsilon_2(p^{k+1},u^{k+1})\|^2\\
        &\ +2a_{k+1}-\|z^{k+1}-z^k\|^2_{S_2}+\|z^k-z^{k-1}\|^2_{S_2}\\
        &\ +2b_{k+1}-\|p^{k+1}-p^k\|^2_{S_3}+\|p^k-p^{k-1}\|^2_{S_3}\\
        &\ -\alpha\|z^{k+1}-z^k\|^2-\alpha\|z^{k+1}_e\|^2+\alpha\|z^k_e\|^2\\
        &\ -\gamma\|p^{k+1}-p^k\|^2-\gamma\|p^{k+1}_e\|^2+\gamma\|p^k_e\|^2\\
        &\ -\|u^{k+1}-u^k\|^2_{S_1}-\|u^{k+1}_e\|^2_{S_1}+\|u^k_e\|^2_{S_1}\\
        &\ -\|z^{k+1}-z^k\|^2_{S_2}-\|z^{k+1}_e\|^2_{S_2}+\|z^k_e\|^2_{S_2}\\
        &\ -\|p^{k+1}-p^k\|^2_{S_3}-\|p^{k+1}_e\|^2_{S_3}+\|p^k_e\|^2_{S_3}.
    \end{aligned}
\label{inequ}
\end{equation}
For convenience, we define
\begin{equation}
    \left\{
    \begin{aligned}
        \delta_{k+1}&=\min\{\tau,1+\tau-\tau^2\}(\alpha\|z^{k+1}-z^k\|^2\\
        &+\gamma\|p^{k+1}-p^k\|^2) +\|z^{k+1}-z^k\|^2_{S_2}+\|p^{k+1}-p^k\|^2_{S_3}\\
        t_{k+1}& =\delta_{k+1}+\|u^{k+1}-u^k\|^2_{S_1}+2\|z^{k+1}_e\|^2_{\Sigma_F}+2\|p^{k+1}_e\|^2_{\Sigma_G},\\
        \psi_{k+1}& =\theta(u^{k+1},z^{k+1},p^{k+1},w^{k+1},q^{k+1})+\|z^{k+1}-z^k\|^2_{S_2}\\
        &\ +\|p^{k+1}-p^k\|^2_{S_3},
    \end{aligned}
    \right.
\label{defs}
\end{equation}
where $\theta(u,z,p,w,q)=(\tau\alpha)^{-1}\|w-w^*\|^2+(\tau\gamma)^{-1}\|w-w^*\|^2+(\tau\gamma)^{-1}\|q-q^*\|^2+\|u-u^*\|^2_{S_1}+\|z-z^*\|^2_{S_2}+\|p-p^*\|^2_{S_3}+\alpha\|z-z^*\|^2+\gamma\|p-p^*\|^2.$ We consider two cases $\tau\in(0,1]$ and $\tau>1$ respectively.\\
Case 1: $\tau\in(0,1]$. Note that
\begin{equation}
    \begin{aligned}
        2\langle\epsilon_1(z^k,u^k),z^{k+1}-z^k\rangle\leq\|z^{k+1}-z^k\|^2+\|\epsilon_1(z^k,u^k)\|^2,\\
        2\langle\epsilon_2(p^k,u^k),p^{k+1}-p^k\rangle\leq\|p^{k+1}-p^k\|^2+\|\epsilon_2(p^k,u^k)\|^2.
    \end{aligned}
\label{aux-est1}
\end{equation}
By the definition of $a_{k+1},b_{k+1}$ and (\ref{inequ}), we have
\begin{equation}
    \begin{aligned}
        &t_{k+1}+\alpha\|\epsilon_1(z^{k+1},u^{k+1})\|^2+\gamma\|\epsilon_2(p^{k+1},u^{k+1})\|^2\\
        &\leq[\psi_k+(1-\tau)\alpha\|\epsilon_1(z^k,u^k)\|^2+(1-\tau)\gamma\|\epsilon_2(p^k,u^k)\|^2]\\
        &-[\psi_{k+1}+(1-\tau)\alpha\|\epsilon_1(z^{k+1},u^{k+1})\|^2\\
        &+(1-\tau)\gamma\|\epsilon_2(p^{k+1},u^{k+1})\|^2].
    \end{aligned}
\label{bound1}
\end{equation}
Case 2: $\tau>1$. Similarly, we have
\begin{equation}
    \begin{aligned}
        -2\langle\epsilon_1(z^k,u^k),z^{k+1}-z^k\rangle\leq\tau\|z^{k+1}-z^k\|^2+\tau^{-1}\|\epsilon_1(z^k,u^k)\|^2,\\
        -2\langle\epsilon_2(p^k,u^k),p^{k+1}-p^k\rangle\leq\tau\|p^{k+1}-p^k\|^2+\tau^{-1}\|\epsilon_2(p^k,u^k)\|^2.
    \end{aligned}
\label{aux-est2}
\end{equation}
Then
\begin{equation}
    \begin{aligned}
        &t_{k+1}+\tau^{-1}(1+\tau-\tau^2)[\alpha\|\epsilon_1(z^{k+1},u^{k+1})\|^2\\
        &+\gamma\|\epsilon_2(p^{k+1},u^{k+1})\|^2]\\
        &\leq[\psi_k+(1-\tau^{-1})\alpha\|\epsilon_1(z^k,u^k)\|^2+(1-\tau^{-1})\gamma\|\epsilon_2(p^k,u^k)\|^2]\\
        &-[\psi_{k+1}+(1-\tau^{-1})\alpha\|\epsilon_1(z^{k+1},u^{k+1})\|^2\\
        &+(1-\tau^{-1})\gamma\|\epsilon_2(p^{k+1},u^{k+1})\|^2].
    \end{aligned}
\label{bound2}
\end{equation}
From (\ref{defs}-\ref{bound2}), we see that $\psi_{k+1}$ is bounded and
\begin{equation}
    \begin{aligned}
        &\lim_{k\rightarrow\infty}t_{k+1}=0,\\
        &\lim_{k\rightarrow\infty}\|w^{k+1}-w^k\|=\lim_{k\rightarrow\infty}(\tau\alpha)^{-1}\|\epsilon_1(z^{k+1},u^{k+1})\|=0,\\
        &\lim_{k\rightarrow\infty}\|q^{k+1}-q^k\|=\lim_{k\rightarrow\infty}(\tau\gamma)^{-1}\|\epsilon_2(p^{k+1},u^{k+1})\|=0.
    \end{aligned}
\label{lim1}
\end{equation}
By the definition of $\theta_{k+1}$ and $t_{k+1}$, the sequences $\{\|w^{k+1}\|\}$, $\{\|q^{k+1}\|\}$, $\{\|u^{k+1}_e\|_{S_1}\}$, $\{\|z^{k+1}_e\|^2_{\Sigma_F+S_2+\alpha I}\}$, $\{\|p^{k+1}_e\|^2_{\Sigma_G+S_3+\gamma I}\}$ are all bounded.
Since $\Sigma_F+S_2+\alpha I\succ0$, $\Sigma_G+S_3|\gamma I\succ0$,
$\|z^{k+1}_e\|$, $\|p^{k+1}_e\|$ are all bounded. Since $S_1$ is selected to be positive definite, $\|u^{k+1}_e\|$ is also bounded.\\
Hence, the sequence $\{(u^k,z^k,p^k,w^k,q^k)\}$ is bounded. Therefore, there is a subsequence converging to a cluster point.\\
Suppose $$\{(u^{k_i},z^{k_i},p^{k_i},w^{k_i},q^{k_i})\}\rightarrow(u^\infty,z^\infty,p^\infty,w^\infty,q^\infty).$$
From (\ref{lim1}), we have
\begin{equation}
    \begin{aligned}
        &\lim_{k\rightarrow\infty}\|z^{k+1}-z^k\|=0,\ \lim_{k\rightarrow\infty}\|p^{k+1}-p^k\|=0,\\
        &\lim_{k\rightarrow\infty}\|w^{k+1}-w^k\|=0,\ \lim_{k\rightarrow\infty}\|q^{k+1}-q^k\|=0,\\
        &\lim_{k\rightarrow\infty}\|u^{k+1}-u^k\|^2_{S_1}=0,\ \lim_{k\rightarrow\infty}\|z^{k+1}-z^k\|^2_{S_2}=0,\ \\
        & \lim_{k\rightarrow\infty}\|p^{k+1}-p^k\|^2_{S_3}=0. 
    \end{aligned}
\label{lim2}
\end{equation}
Note that since
\begin{equation}
    \begin{aligned}
        \|z^k-Au^{k+1}\|\leq\|z^{k+1}-Au^{k+1}\|+\|z^{k+1}-z^k\|,\\
        \|p^k-\nabla u^{k+1}\|\leq\|p^{k+1}-\nabla u^{k+1}\|+\|p^{k+1}-p^k\|.
    \end{aligned}
\end{equation}
We get from (\ref{lim1}) and (\ref{lim2}) that
\begin{equation}
    \lim_{k\rightarrow\infty}\|z^k-Au^{k+1}\|=0,\ \lim_{k\rightarrow\infty}\|p^k-\nabla u^{k+1}\|=0.
\label{lim3}
\end{equation}
Taking limit on both sides of (\ref{spadmm-condition}) and by the closedness of $\partial F$, $\partial G$ \cite{borwein2010convex}, we have 
\begin{equation}
\begin{aligned}
    &-A^*w^\infty-\nabla^Tq^\infty=0,\\
    &-w^\infty\in\partial G(z^\infty),\\
    &-q^\infty\in\partial F(p^\infty),\\
    &z^\infty-Au^\infty=0,\ p^\infty-\nabla u^\infty=0.
\end{aligned}
\label{lim-kkt}
\end{equation}
Therefore, $(u^\infty,z^\infty,p^\infty)$ is an optimal solution to (\ref{model:relaxed discrete model}), and $(w^\infty,q^\infty)$ are the corresponding Lagrange multipliers.
We now show the convergence of the whole sequence. Since $(u^\infty,z^\infty,p^\infty,w^\infty,q^\infty)$ satisfies (\ref{kkt-condition}), we replace $(u^*,z^*,p^*,w^*,q^*)$ in the above by $(u^\infty,z^\infty,p^\infty,w^\infty,q^\infty)$.
Therefore, for $\tau\in(0,1]$, $\{\psi_{k_i}+(1-\tau)[\alpha\|\epsilon_1(z^{k_i},u^{k_i})\|^2+\gamma\|\epsilon_2(p^{k_i},u^{k_i})\|^2]\}\rightarrow0$ and for $\tau\in(1,(1+\sqrt{5})/2)$, $\{\psi_{k_i}+(1-\tau^{-1})[\alpha\|\epsilon_1(z^{k_i},u^{k_i})\|^2+\gamma\|\epsilon_2(p^{k_i},u^{k_i})\|^2]\}\rightarrow0$. Since these 2 subsequences are from a non-increasing sequence, we have both $\{\|\epsilon_1(z^k,u^k)\|\}$, $\{\|\epsilon_2(p^k,u^k)\|\}$ converge to $0$. Consequently, we have $\lim_{k\rightarrow\infty}\psi_k=0$. Therefore,
$\lim_{k\rightarrow\infty}w^k=w^\infty$, and $\lim_{k\rightarrow\infty}q^k=q^\infty$.
Hence, combining with (\ref{lim1}), we have 
\begin{equation}
    \begin{aligned}
        \lim_{k\rightarrow\infty}\|u^k_e\|^2_{S_1}=0,\\
        \lim_{k\rightarrow\infty}\|z^k_e\|^2_{\Sigma_F+S_2+\alpha I}=0,\\
        \lim_{k\rightarrow\infty}\|p^k_e\|^2_{\Sigma_G+S_3+\gamma I}=0.
    \end{aligned}
\label{final-lim}
\end{equation}
Since $S_1$, $\Sigma_F+S_2+\alpha I$ and $\Sigma_G+S_3+\gamma I$ are positive definite, we have $\lim_{k\rightarrow\infty}u^k=u^\infty$, $\lim_{k\rightarrow\infty}z^k=z^\infty$ and $\lim_{k\rightarrow\infty}p^k=p^\infty$.
\end{appendix}


 
\section*{References Section}
\bibliographystyle{IEEEtran}
\bibliography{mybibfile}

\end{document}